\documentclass{elsarticle}
 
\usepackage{tikz}
\usetikzlibrary{cd}

\usepackage{amsmath,amssymb,amsfonts,mathrsfs,mathtools, braket}
\usepackage{enumerate}
\newtheorem{theorem}{Theorem}[section]
\newtheorem{lemma}[theorem]{Lemma}
\newtheorem{corollary}[theorem]{Corollary}
\newtheorem{proposition}[theorem]{Proposition}

\newproof{proof}{Proof}
\newdefinition{definition}[theorem]{Definition}
\newdefinition{example}[theorem]{Example}
\newdefinition{remark}[theorem]{Remark}
\newdefinition{assumption}[theorem]{Assumption}

\numberwithin{equation}{section}

\newcommand{\reftext}[1]{#1}

\newcommand{\calP}{\mathcal{P}}
\newcommand{\calQ}{\mathcal{Q}}

\newcommand{\inv}{{-1}}

\allowdisplaybreaks
\begin{document}

\title{Classifying decomposition and wavelet coorbit spaces using coarse geometry}

\author{Hartmut F\"uhr}
\address{Lehrstuhl A f\"ur Mathematik, RWTH Aachen, 52056 Aachen, Germany}
\ead{fuehr@matha.rwth-aachen.de}

\author{Ren\'e Koch}
\address{Lehrstuhl A f\"ur Mathematik, RWTH Aachen, 52056 Aachen, Germany}
\ead{rene.koch@matha.rwth-aachen.de}

\begin{abstract}
This paper is concerned with the study of Besov-type decomposition spaces,
which are scales of spaces associated to suitably defined coverings of
the euclidean space $\mathbb{R}^{d}$, or suitable open subsets thereof.
A fundamental problem in this domain, that is currently not fully understood,
is deciding when two different coverings give rise to the same scale of
decomposition spaces.

In this paper, we establish a coarse geometric approach to this problem,
and show how it specializes for the case of wavelet coorbit spaces associated
to a particular class of matrix groups $H < GL(\mathbb{R}^{d})$ acting
via dilations. These spaces can be understood as special instances of decomposition
spaces, and it turns out that the question whether two different dilation
groups $H_{1},H_{2}$ have the same coorbit spaces can be decided by investigating
whether a suitably defined map $\phi : H_{1} \to H_{2}$ is a quasi-isometry
with respect to suitably defined word metrics. We then proceed to apply
this criterion to a large class of dilation groups called \emph{shearlet
dilation groups}, where this quasi-isometry condition can be characterized
algebraically. We close with the discussion of selected examples.
\end{abstract}
%
%
%
%
\maketitle
\section{Introduction}
\label{sec:introduction}

This paper is concerned with the study of two related classes of function
spaces on euclidean spaces, \textit{wavelet coorbit spaces} and
\textit{decomposition spaces}. Both classes can be understood as rather
sweeping generalizations of Besov spaces, and their inception dates back
to work by Feichtinger, Gr\"ochenig and Gr\"obner (e.g.
\cite{FeichtingerGroebnerBanachSpacesOfDistributions,FeiGroeAUniefApprToIntegrGroupRepAndTheirAtomicDec,FeiGroeBanSpaRelToIntGrpRepI,FeiGroeBanSpaRelToIntGrpRepII})
in the 1980s. We refer to \cite{MR3682610} for more background information
on this development, and in particular for an explanation of the relationship
of these notions to the influential work of Frazier and Jawerth
\cite{MR808825}.

The main thrust of the initial work was to provide a unified view onto
a variety of function spaces, specifically Besov spaces and modulation
spaces, using notions from abstract harmonic and Fourier analysis. Coorbit
spaces are defined using group representation theory
\cite{FeiGroeBanSpaRelToIntGrpRepI,FeiGroeBanSpaRelToIntGrpRepII}, and
special instances of coorbit spaces have been studied in a variety of settings
and for various classes of underlying locally compact groups in recent
years; see e.g.
\cite{DahlkeShearletCoorbitSpacesCompactlySupported,DahlkeShearletCoorbitSpacesAssociatedToFrames,FuehCooSpaAndWavCoeDecOveGenDilGro,FuRT,FePa,fischer2018heisenbergmodulation}
for a small sample of the literature. The class of
\textit{admissible dilation groups} underlying higher dimensional wavelet
coorbit spaces turns out to be a particularly rich source of examples
\cite{FuehCooSpaAndWavCoeDecOveGenDilGro}, see
\begin{itemize}
\item \cite{FuDiss} for a classification of admissible dilation groups
in dimension two up to conjugacy;
\item \cite{FuehrContinuousWaveletTransformsWithAbelian} for the systematic
construction and classification (up to conjugacy) of abelian admissible
dilation groups;
\item
\cite{DahlkeShearletCoorbitSpacesCompactlySupported,DahlkeShearletCoorbitSpacesAssociatedToFrames}
for the initial definition of shearlet dilation groups in dimension two
and higher;
\item \cite{DahlkeHaeuTesCooSpaTheForTheToeSheTra} for the introduction
of \textit{Toeplitz shearlet dilation groups} in dimension three and higher;
\item \cite{FuRT,AlbertiEtAl2017} for the introduction and systematic
construction of generalized shearlet dilation groups in arbitrary dimensions
(comprising the previous examples);
\item \cite{FuehrCurreyOussa} for a classification of admissible dilation
groups in dimension three, up to conjugacy.
\end{itemize}
We emphasize that the items on this list do not refer to single dilation
groups, but rather to classes of groups, most of which tend to become quite
large for higher dimensions.\looseness=1

By contrast, decomposition spaces are of a more Fourier analytic nature.
Whereas coorbit spaces depend on the choice of a particular group representation,
the construction of decomposition spaces departs from a certain covering
of the frequency domain. After their inception in
\cite{FeichtingerGroebnerBanachSpacesOfDistributions}, the theory of decomposition
spaces lay somewhat dormant for a while, but their usefulness in capturing
approximation-theoretic features of anisotropic systems such as curvelets
was realized later by Borup and Nielsen, see e.g.
\cite{BorupNielsenFrameDecompositionOfDecOfSpaces,LabateShearletSmoothnessSpaces}.
A further major boost to the subject was provided by the transformative
work of Voigtlaender
\cite{Voigtlaender2015PHD,VoigtlaenderEmbeddingsOfDecompositionSpaces},
which is also the foundation of the current paper. At this point in time,
large classes of function spaces have been recognized as special instances
of the decomposition space formalism, as the following list shows:
\begin{itemize}
\item $\alpha $-modulation spaces \cite{Groebner1992PhD};
\item curvelet smoothness spaces
\cite{BorupNielsenFrameDecompositionOfDecOfSpaces};
\item shearlet smoothness spaces
\cite{LabateShearletSmoothnessSpaces};
\item homogeneous and inhomogeneous \textit{anisotropic} Besov spaces
\cite{Bo,FuCh};
\item wave packet smoothness spaces \cite{BytVoi};
\item wavelet coorbit spaces associated to admissible dilation group, i.e.,
\textit{all examples on the previous list}
\cite{FuehrVoigtlWavCooSpaViewAsDecSpa});
\item wavelet coorbit spaces associated to
\textit{integrably admissible dilation groups}, see
\cite{fuehrvelthoven2020coorbit}. This class properly contains the class
mentioned in the previous item, but also the class of homogeneous anisotropic
Besov spaces.
\end{itemize}
Again, most of the items on this list represent classes or families of
constructions rather than a single one, and these classes can become quite
large with increasing dimension.

While coorbit and decomposition spaces rest on somewhat distinct mathematical
foundations, the resulting spaces share an important common approximation-theoretic
interpretation. What is common both to coorbit and decomposition spaces
is the existence of a family of vectors
$(\eta _{x})_{x \in \mathcal{X}}$ in a suitable Hilbert space
$\mathcal{H}$ that acts as a \textit{continuous frame} of
$\mathcal{H}$, i.e., guaranteeing a norm equivalence
%
\begin{equation}
\label{eqn:transform_norm_equiv}
\| f \|_{\mathcal{H}} \asymp \|(\langle f, \eta _{x} \rangle )_{x
\in \mathcal{X}}\|_{L^{2}(\mathcal{X},d\mu )}~,
\end{equation}
with respect to a suitably chosen measure $\mu $ on $\mathcal{X}$, as well
as an associated \textit{inversion formula}
\begin{equation*}
f = \int _{\mathcal{X}} \langle f, \eta _{x} \rangle \widetilde{\eta}_{x}
d\mu (x)~,
\end{equation*}
with a suitably chosen \textit{dual system}
$(\widetilde{\eta}_{x})_{x \in \mathcal{X}} \subset \mathcal{H}$. The operator
\begin{equation*}
V_{\eta }: f \mapsto (\langle f, \eta _{x} \rangle )_{x \in
\mathcal{X}}
\end{equation*}
mapping $f$ to its expansion coefficients can be understood as a generalized
wavelet transform. Typically, the parameter space $\mathcal{X}$ provides
an interpretation of the elements $\eta _{x}$ of the continuous frame as
basic building blocks, and the inversion formula expresses $f$ as a continuous
expansion in these building blocks.

For concreteness, let us quickly sketch the case of generalized wavelet
systems: Given a dilation group $H< GL(\mathbb{R}^{d})$, we let
$\mathcal{X} = \mathbb{R}^{d} \rtimes H$ denote the subgroup of the full
affine group of $\mathbb{R}^{d}$ generated by the translations in
$\mathbb{R}^{d}$ and $H$, and define the quasi-regular representation
\begin{equation*}
\pi : \mathbb{R}^{n} \rtimes H \to \mathcal{U}(L^{2}(\mathbb{R}^{n}))~,~
\left ( \pi (x,h) f \right ) (y) = |{\mathrm{det}}(h)|^{-1/2} f(h^{-1}(y-x))~,
\end{equation*}
and for suitably chosen $\eta \in L^{2}(\mathbb{R}^{n})$, we let
\begin{equation*}
\eta _{(x,h)} = \pi (x,h) \eta ~.
\end{equation*}
The element $(x,h) \in \mathcal{X}$ corresponds to a wavelet
$\eta _{(x,h)}$ centered at $x$ and scaled by $h \in H$. Very often, explicit
parametrizations of $H$ allow to interpret the scaling variable $h$ in
$(x,h)$ (and thus $\eta _{(x,h)}$) further, e.g. as a combination of isotropic
scaling and rotation in the case of the similitude group
\cite{AnMuVa}, or as a combination of anisotropic scaling and shearing
in the case of a shearlet dilation group
\cite{DahlkeShearletCoorbitSpacesAssociatedToFrames}. In this group-theoretic
context, the norm equivalence \reftext{(\ref{eqn:transform_norm_equiv})} becomes
a norm \textit{equality}, when $\eta $ is chosen as an
\textit{admissible vector} and $\mu $ is chosen as the left Haar measure
on $G$, and the inversion formula holds with
$\widetilde{\eta}_{(x,h)} = \eta _{(x,h)}$.

One can then proceed to define norms on elements of
$L^{2}(\mathbb{R}^{n})$ by imposing integrability conditions that are more
stringent than the $L^{2}$-norm entering in the Hilbert space norm equivalence
\reftext{(\ref{eqn:transform_norm_equiv})}, i.e., by introducing weights
$\nu $ and considering integrability exponents $p<2$, leading to considering
norms of the type
\begin{equation*}
\| V_{\eta }f \|_{L^{p}_{\nu}}~.
\end{equation*}
Furthermore, discretization results allow to replace the (typically continuously
indexed) systems $(\eta _{x})_{x \in \mathcal{X}}$ by suitably chosen discrete
subsystems -- in fact, (quasi-)Banach frames -- corresponding to discrete
subsets $\mathcal{X}_{d} \subset \mathcal{X}$, while preserving norm equivalences
such as
\begin{equation*}
\| V_{\eta }f \|_{L^{p}_{\nu}} \asymp \| V_{\eta }f |_{\mathcal{X}_{d}}
\|_{\ell ^{p}_{\nu}}~,
\end{equation*}
see e.g.
\cite{FeichtingerGroebnerBanachSpacesOfDistributions,FeiGroeAUniefApprToIntegrGroupRepAndTheirAtomicDec,FeiGroeBanSpaRelToIntGrpRepI,FeiGroeBanSpaRelToIntGrpRepII}.
For $p<2$ and constant weights this norm equivalence attains additional relevance: Here the
discretization results of the just cited references allow to conclude, that vectors having $p$-summable frame coefficients
have a nontrivial non-linear approximation rate with respect to the frame,
with the decay rate increasing as $p$ decreases. In this way, both coorbit
and decomposition spaces associated to integrability exponents
$p,q <2$ have natural interpretations as spaces of sparse signals with
respect to the respective systems of building blocks. These spaces therefore
capture the approximation-theoretic properties of the building blocks.

Much of the foundational work in
\cite{FeichtingerGroebnerBanachSpacesOfDistributions,FeiGroeAUniefApprToIntegrGroupRepAndTheirAtomicDec,FeiGroeBanSpaRelToIntGrpRepI,FeiGroeBanSpaRelToIntGrpRepII}
was devoted to proving that the definitions were \textit{consistent}, i.e.,
essentially independent of various design choices that enter into the construction
of the systems. In the case of wavelet coorbit spaces, these results were
mostly related to the proper choice of the \textit{analyzing wavelet}
$\eta $ entering the definition of the wavelet system
$(\eta _{(x,h)})_{(x,h) \in \mathcal{X}} = (\pi (x,h) \eta )_{(x,h)
\in \mathcal{X}}$.

By contrast, the influence of the \textit{primary} choices in the design
of the system of building blocks, i.e., the choice of dilation group in
the case of generalized wavelet systems, and the choice of covering in
the case of decomposition spaces, is much less understood. Generally speaking,
one very much expects that qualitatively different primary choices will
result in different approximation-theoretic properties of the building
blocks; for systems such as curvelets or shearlets, this expectation was
the driving factor for their inception. That said, the understanding what
``qualitatively different'' actually means in this context is obviously
an important part of this discussion, and it is currently not fully developed.

Prior to the work of Voigtlaender, with few exceptions such as
\cite{LabateShearletSmoothnessSpaces}, most of the work on coorbit or decomposition
spaces tended to stay within the confines provided by a fixed primary choice,
i.e., tended to concentrate on, say, $\alpha $-modulation spaces, or on
shearlet coorbit spaces. Voigtlaender's publications since 2015, specifically
\cite{Voigtlaender2015PHD,FuehrVoigtlWavCooSpaViewAsDecSpa,VoigtlaenderEmbeddingsOfDecompositionSpaces},
systematically develop techniques that allow to cross these boundaries
between different classes of spaces, by providing sharp embedding results
between decomposition spaces associated to different coverings, or into
classical smoothness spaces.

By contrast to the scope of the mentioned papers
\cite{FuehrVoigtlWavCooSpaViewAsDecSpa,Voigtlaender2015PHD,VoigtlaenderEmbeddingsOfDecompositionSpaces},
the aims of this paper are somewhat more modest and elementary. We address
the following fundamental question, both for decomposition spaces and for
wavelet coorbit spaces: When do two initially different primary design
choices (e.g., different coverings or dilation groups) result in the
\textit{same} scales of associated spaces? Our central contribution essentially
amounts to rewriting the pertinent criteria formulated in
\cite{VoigtlaenderEmbeddingsOfDecompositionSpaces} by introducing a novel
ingredient to the discussion, namely \textit{coarse geometry}, and demonstrating
that this new perspective can be used effectively.

\subsection{Structure and overview of the paper}
\label{sec1.1}

This paper is based on parts of the PhD thesis
\cite{KochDoktorarbeit} by the second author, with some results substantially
expanded.

As mentioned at the end of the previous subsection, we address two main
questions:
\begin{enumerate}
\item When do two coverings result in the same scale of decomposition spaces?
\item When do two dilation groups possess the same scale of coorbit spaces?
\end{enumerate}
As pointed out in the previous subsection, Question 2 is a special case
of Question 1, and both are essentially answered by results in
\cite{VoigtlaenderEmbeddingsOfDecompositionSpaces}. However, as with many
results in this domain, the precise criteria are somewhat unintuitive,
and tend to be cumbersome to apply in concrete settings. It is one of the
main contributions of our paper to rewrite the criteria in terms of suitable
metrics, following an initial observation made in
\cite{FeichtingerGroebnerBanachSpacesOfDistributions}, and to demonstrate
the usefulness of this reformulation with the help of various examples.

We now give an overview of the paper. The following summary glosses over
various technicalities (such as additional conditions on coverings, or
matters related to the question how coorbit spaces associated to different
groups can be understood to coincide, or the role of connectedness issues),
that are explained in more detail in the subsequent text.

Section~\ref{sec:preliminaries} recalls the fundamentals of decomposition
spaces and coorbit spaces, including the definitions of the various spaces,
and their basic properties. Decomposition spaces, as used throughout the
paper, rely on the notion of \emph{(structured) admissible covering}. We
will concentrate on decomposition spaces associated to weighted mixed
$L^{p}$-norms, i.e., to inner norms $\| \cdot \|_{L^{p}}$ and outer norms
$\| \cdot \|_{\ell ^{q}_{v}}$. The main takeaway from subsection \ref{subsect:dspace} is the definition of the notion of
\textit{weak equivalence} of coverings, and its relevance for the associated
decomposition spaces. \reftext{Lemmas~\ref{cor:WeakEquivalenceOfInducedCoverings} and \ref{lem:weak_equiv_suff}} show that weak equivalence of two admissible
coverings is equivalent to the fact that the associated scales of decomposition
spaces coincide. They also formulate a rigidity result stating that if
two decomposition spaces associated to different coverings coincide for
\textit{some} pair of integrability exponents $(p,q) \neq (2,2)$, then
the full scales of decomposition spaces agree. A complementary result that
is of particular relevance to wavelet coorbit spaces addresses the question
when dual coverings of \textit{different} sets give rise to identical decomposition
spaces, see \reftext{Theorem~\ref{thm:different_freq_supp}}.

We then recount the relevant results of wavelet coorbit theory in higher
dimensions, in particular the notion of admissible dilation groups
$H$ and their unique open dual orbits $\mathcal{O} = H^{T} \xi $, for suitably
chosen $\xi \in \mathbb{R}^{d}$. The decomposition space description of
the coorbit spaces associated to the quasi-regular representation of
$\mathbb{R}^{d} \rtimes H$ is obtained via the so-called
\textit{induced covering} arising from the dual action of $H$ on
$\mathcal{O}$, see \reftext{Theorem~\ref{thm:FourierIsoCoorbitDecSpaces}}. We then
introduce the notion of \textit{coorbit equivalence} for two dilation groups
$H_{1},H_{2}$, expressing when their coorbit spaces coincide; see \reftext{Definition~\ref{defn:coorbit_equivalent}}. We combine \reftext{Theorem~\ref{thm:FourierIsoCoorbitDecSpaces}} with the results characterizing when
two dual coverings result in the same scale of decomposition spaces, and
with \reftext{Theorem~\ref{thm:different_freq_supp}}, to obtain an important intermediate
characterization of coorbit equivalence in \reftext{Theorem~\ref{thm:coorbit_equiv_dual_orbits}}.

Section~\ref{sect:metric_reformulation} then proceeds to introduce metric
language to the discussion. We first give a short review of the relevant
notions from coarse geometry in Subsection \ref{subsect:elem_coarse}. A
covering $\mathcal{P}$ of a frequency set $\mathcal{O}$ allows to introduce,
in a very natural way, a related metric $d_{\mathcal{P}}$, and it turns
out that weak equivalence of $\mathcal{P}$ to a second covering
$\mathcal{Q}$ of the same set $\mathcal{O}$ is equivalent to the fact that
$id_{\mathcal{O}}: (\mathcal{O}, d_{\mathcal{P}}) \to (\mathcal{O},d_{
\mathcal{Q}})$ is a quasi-isometry, or equivalently, a coarse equivalence;
see \reftext{Theorem~\ref{thm:WeakEquivQuasiIso}}.

As already mentioned, it is important to point out that the definition
of the metric goes back to the original source
\cite{FeichtingerGroebnerBanachSpacesOfDistributions}, and Proposition
3.8 of the mentioned paper is a relevant precursor of \reftext{Theorem~\ref{thm:WeakEquivQuasiIso}}. However, weak equivalence of coverings was
initially only known to be a \emph{sufficient} criterion for the property
that two coverings define the same scale of decomposition spaces; the converse
was proved later in \cite{Voigtlaender2015PHD}.

Section~\ref{sect:metric_coorbit} transfers the results obtained for decomposition
spaces to the coorbit setting. We begin by reviewing basic notions from
metric group theory in Subsection \ref{subsect:word_metrics}. Subsection \ref{subsect:CoarseEquivalenceOfOrbitAndGroup} establishes that the canonical
projection map from the group onto the dual orbit is a quasi-isometry.
Subsection \ref{subsect:cc_eq} contains one of the main results of this
paper, namely \reftext{Theorem~\ref{thm:MainTheoremEquivalenceOfGroups}}, which combines
the observations from Subsection \ref{subsect:CoarseEquivalenceOfOrbitAndGroup} with the criteria from the
previous sections to formulate a metric criterion for coorbit equivalence:
Two admissible groups $H_{1}$ and $H_{2}$ can only be coorbit equivalent
if their dual orbits coincide, i.e.,
$\mathcal{O} = H_{1}^{T} \xi _{0} = H_{2}^{T} \xi _{0}$ for some suitable
$\xi _{0} \in \mathbb{R}^{d}$. If that condition is fulfilled, consider
the two associated canonical projections
$p_{\xi _{0}}^{H_{1}} : H_{1} \ni h \mapsto h^{-T} \xi _{0} \in
\mathcal{O}$, and $p_{\xi _{0}}^{H_{2}} : H_{2} \to \mathcal{O}$ defined
analogously. Let
$(p_{\xi _{0}}^{H_{2}})^{-1} : \mathcal{O} \to H_{2}$ denote any right inverse
of $p_{\xi _{0}}^{H_{2}}$. Then $H_{1}$ and $H_{2}$ are coorbit equivalent
if and only if
$(p_{\xi _{0}}^{H_{2}})^{-1} \circ p_{\xi _{0}}^{H_{1}}: (H_{1},d_{H_{1}})
\to (H_{2},d_{H_{2}})$ is a quasi-isometry with respect to any choice of
word metrics $d_{H_{1}},d_{H_{2}}$ on $H_{1},H_{2}$ associated to suitable
relatively compact neighborhoods of unity. Again, the quasi-isometry property
is equivalent to coarse equivalence. Thus coorbit equivalence has been
translated to a concrete geometric group theory problem.

Section~\ref{sect:coorbit_equiv_shearlet} contains an illustration that
this translation can be put to use for the characterization of coorbit
equivalence within a large example class, namely that of generalized shearlet
dilation groups. The main result of that section is \reftext{Theorem~\ref{thm:char_ce_shearlet}}, which contains a concise characterization of
coorbit equivalence for shearlet dilation groups. The fact that this equivalence
leads to rather stringent conditions on the groups under consideration
emphasizes the richness of coorbit theory in higher dimensions. The results
indicate that in higher dimensions, coorbit equivalence of different groups
is a fairly rare phenomenon. While this observation may have been expected,
it is the main contribution of this paper to provide sharp criteria, methods
for their verification, and explicit classes of groups corroborating this
expectation.

The closing section contains various examples that further illustrate the
usefulness of the approach developed in this paper.

\section{Fundamentals of coorbit and decomposition spaces}
\label{sec:preliminaries}

\subsection{Decomposition spaces}
\label{subsect:dspace}

The starting point for the definition of decomposition spaces is the notion
of an \emph{admissible covering} $\mathcal{Q}=(Q_{i})_{i\in I}$ of some
open set $\mathcal{O} \subset \mathbb{R}^{d}$ (see \cite{FeichtingerGroebnerBanachSpacesOfDistributions}),
which is a family of nonempty sets $Q_{i} \subset \mathbb{R}^{d}$ such
that
\begin{enumerate}[ii)]
\item[i)] $\bigcup _{i\in I} Q_{i} = \mathcal{O}$ and
\item[ii)]
$\sup _{i\in I} \sharp \{j\in I: Q_{i} \cap Q_{j} \neq \emptyset \}<
\infty $.
\end{enumerate}
Throughout this paper, we will concentrate on the class of \emph{(tight)
structured admissible covering}, see Definition 2.5 of
\cite{VoigtlaenderEmbeddingsOfDecompositionSpaces}. This means that
$Q_{i} = T_{i} Q + b_{i}$ with
$T_{i}\in \mathrm{GL}(\mathbb{R}^{d})$, $b_{i}\in \mathbb{R}^{d}$ with
an open, precompact set $Q$, and the involved matrices fulfill
%
\begin{equation}
\label{eqn:str_cov_norm}
\sup _{i,j \in I : Q_{j} \cap Q_{j} \neq \emptyset} \| T_{i}^{-1} T_{j}
\| < \infty .
\end{equation}

The next ingredient in the definition of decomposition spaces is a special
partition of unity $\Phi =(\varphi _{i})_{i\in I}$ subordinate to
$\mathcal{Q}$, also called $\mathrm{L}^{p}$-BAPU (bounded admissible partition
of unity), with the following properties
\begin{enumerate}[iii)]
\item[i)]
$\varphi _{i} \in C_{c}^{\infty}(\mathcal{O})\quad \forall i\in I$,
\item[ii)]
$\sum _{i\in I} \varphi _{i}(x)=1 \quad \forall x\in \mathcal{O}$,
\item[iii)] $\varphi _{i}(x)=0$ for $x\in \mathbb{R}^{d}\setminus Q_{i}$ and
$i\in I$,
\item[iv)] if $1\leq p \leq \infty $:
$\sup _{i\in I}\Vert \mathcal{F}^{-1} \varphi _{i}\Vert _{\mathrm{L}^{1}}<
\infty $,%

if $0<p<1$:\quad
$\sup _{i\in I}|\det (T_{i})|^{\frac{1}{p}-1}\Vert \mathcal{F}^{-1}
\varphi _{i}\Vert _{\mathrm{L}^{p}}<\infty $.
\end{enumerate}
Here, $\mathcal{F}$ denotes the usual Fourier transform of a function in
$\mathrm{L^{2}}(\mathbb{R}^{d})$ defined by
\begin{align*}
\mathcal{F}f(\xi ):= \int _{\mathbb{R}^{d}} f(x)e^{-2\pi i\langle x,
\xi \rangle}\mathrm{d}x
\end{align*}
for $\xi \in \mathbb{R}^{d}$. We also use the notation
$\widehat{f}:=\mathcal{F}(f)$. The definition of decomposition spaces requires
one last ingredient, namely a weight $(u_{i})_{i\in I}$ such that there
exists $C>0$ with $u_{i} \leq C u_{j}$ for all
$i,j \in I: Q_{i} \cap Q_{j} \neq \emptyset $. A weight with this property
is also called \emph{$\mathcal{Q}$-moderate}. The interpretation of this
property is that the value of $(u_{i})_{i\in I}$ is comparable for indices
corresponding to sets which are ``close'' to each other. Finally, we define
the
\textit{(Fourier-side) decomposition space with respect to the covering
$\mathcal{Q}$ and the weight $(u_{i})_{i \in I}$ with integrability exponents
$0<p,q\leq \infty $} as
%
\begin{align}
\mathcal{D}(\mathcal{Q}, \mathrm{L}^{p}, \ell ^{q}_{u}):=\{f\in
\mathcal{D}'(\mathcal{O}): \Vert f\Vert _{\mathcal{D}(\mathcal{Q},
\mathrm{L}^{p}, \ell ^{q}_{u})}< \infty \}
\end{align}
for
%
\begin{align}
\Vert f\Vert _{\mathcal{D}(\mathcal{Q}, \mathrm{L}^{p}, \ell ^{q}_{u})}:=
\left \Vert \left (u_{i} \cdot \Vert \mathcal{F}^{-1}(\varphi _{i} f)
\Vert _{\mathrm{L}^{p}(\mathbb{R}^{d})} \right )_{i\in I}\right
\Vert _{\ell ^{q}(I)}.
\end{align}
As the notation suggests, the decomposition spaces are independent of the
precise choice of $\Phi $ \cite[Corollary 3.4.11]{Voigtlaender2015PHD}.

A crucial concept is the definition of the set of neighbors of a covering.

\begin{definition}[\cite{FeichtingerGroebnerBanachSpacesOfDistributions} Definition 2.3]%
\label{def:SetOfNeighbors}
For a covering $\mathcal{Q}= (Q_{i})_{i\in I}$ of $\mathcal{O}$ with
$Q_{i}\subset \mathcal{O}$ for all $i\in I$, we define the
\textit{set of neighbors of a subset $J\subset I$} as
\begin{align*}
J^{*}:=\Set{i\in I| \exists j\in J:\ Q_{i} \cap Q_{j} \neq \emptyset}.
\end{align*}
By induction, we set $J^{0*}:= J$ and
$J^{(n+1)*} = \left (J^{n*}\right )^{*}$ for $n\in \mathbb{N}_{0}$. Moreover,
we use the shorthand notations $i^{k*}:=\Set{i}^{k*}$ and define
$Q_{i}^{k*}:=\bigcup _{j\in i^{k*}}Q_{j}$ for $i\in I$ and
$k\in \mathbb{N}_{0}$.
\end{definition}

In the remaining part of this subsection, we take a look at relations between
different coverings. The ultimate purpose of these relations is the clarification
when two coverings lead to the same scale of decomposition spaces. While
our exposition follows \cite{Voigtlaender2015PHD}, most of the definitions
hark back to \cite{FeichtingerGroebnerBanachSpacesOfDistributions}.%

\begin{definition}[\cite{Voigtlaender2015PHD} Definition 3.3.1.]%
\label{def:NeighborsAndEquivalence}
Let $\mathcal{Q}= (Q_{i})_{i\in I}$ and
$\mathcal{P} = (P_{j})_{j\in J}$ be families of subsets of
$\mathbb{R}^{d}$.
\begin{enumerate}[iii)]
\item[i)] We define the \textit{set of $\mathcal{P}$-neighbors of}
$i\in I $ by
$J_{i}:= \Set{j\in J | Q_{i}\cap P_{j} \neq \emptyset}$. More generally,
we call $J_{i}$ and $I_{j}$ \textit{intersection sets} for the coverings
$\mathcal{Q}$ and $\mathcal{P}$.
\item[ii)] We call $\mathcal{Q}$ \textit{weakly subordinate} to
$\mathcal{P}$ if
$N(\mathcal{Q}, \mathcal{P}):=\sup _{i\in I}\mathopen{\lvert }J_{i}
\mathclose{\rvert } < \infty $.

The quantity $N(\mathcal{P}, \mathcal{Q})$ is defined analogously, and
we call $\mathcal{Q}$ and $\mathcal{P}$ \textit{weakly equivalent} if
$N(\mathcal{P}, \mathcal{Q})<\infty $ and
$N(\mathcal{Q}, \mathcal{P})<\infty $.
\item[iii)] We call $\mathcal{Q}$ \textit{almost subordinate} to
$\mathcal{P}$ if there exists $k\in \mathbb{N}_{0}$ such that for every
$i\in I$ there exists an $j_{i}\in J$ with
$Q_{i}\subset P_{j_{i}}^{k*}$. If $k=0$ is a valid choice, then we call
$\mathcal{Q}$ \textit{subordinate} to $\mathcal{P}$.
\item[iv)] We call $\mathcal{Q}$ \textit{weakly equivalent} to
$\mathcal{P}$ if $\mathcal{Q}$ is weakly subordinate to
$\mathcal{P}$ and $\mathcal{P}$ is weakly subordinate to
$\mathcal{Q}$.
\end{enumerate}
\end{definition}

The relevance of these notions, in particular of weak equivalence, is spelled
out in the next two lemmas. The formulation of the next lemma is a special
case of the cited result.

\begin{lemma}[\cite{VoigtlaenderEmbeddingsOfDecompositionSpaces} Theorem 6.9]%
\label{cor:WeakEquivalenceOfInducedCoverings}
Let $\mathcal{Q}=(Q_{i})_{i\in I}, \mathcal{P}=(P_{j})_{j\in J}$ be two
structured admissible coverings of the open set
$\mathcal{O}\subset \mathbb{R}^{d}$. If $\mathcal{Q}$ and
$\mathcal{P}$ are not weakly equivalent, then
\begin{equation*}
\mathcal{D}(\mathcal{Q}, L^{p}, \ell ^{q}_{u_{1}})\neq \mathcal{D}(
\mathcal{P}, L^{p}, \ell ^{q}_{u_{2}})
\end{equation*}
for all $\mathcal{Q}$-moderate weights $u_{1}:I\to (0,\infty )$, for all
$\mathcal{P}$-moderate weights $u_{2}:J\to (0,\infty )$ and all
$p, q \in (0,\infty ]$ with $(p,q) \neq (2,2)$.
\end{lemma}

The exception $(p,q)\neq (2,2)$ is necessary to exclude trivial cases:
In the case of $(p,q) = (2,2)$ the associated decomposition spaces are
just weighted $\mathrm{L}^{2}$ spaces, by the Plancherel Theorem. In particular,
if $\mathcal{O} \subset \mathbb{R}^{d}$ is open and of full measure, and
the weight $v$ is constant, one finds that
$\mathcal{D}(\mathcal{P},L^{2},\ell ^{2}_{v}) = L^{2}(\mathbb{R}^{d})$,
for all admissible coverings $\mathcal{P}$.

Weak subordinateness and equivalence of coverings are important assumptions
for a multitude of sufficient criteria for embeddings of decomposition
spaces and their equality, as {developed} in
\cite{VoigtlaenderEmbeddingsOfDecompositionSpaces}. The following statement
is \cite[Lemma 6.11]{VoigtlaenderEmbeddingsOfDecompositionSpaces}. The
statement about the range $0 \le p,q \le \infty $ is justified by the remark
following the cited lemma.
%
\begin{lemma}
\label{lem:weak_equiv_suff}
Let $1\leq p,q\leq \infty $ and let
$\emptyset \neq \mathcal{O}\subset \mathbb{R}^{d}$ be open. Further, let
$\mathcal{Q}=(Q_{i})_{i\in I}, \mathcal{P}=(P_{j})_{j\in J}$ be two tight
structured admissible coverings of $\mathcal{O}$, and let $u_{1}$ be a
$\mathcal{Q}$-moderate weight and $u_{2}$ a $\mathcal{P}$-moderate weight.

If $\mathcal{Q}$ and $\mathcal{P}$ are weakly equivalent and there exists
$C>0$ such that $C^{-1}u_{1}(i) \leq u_{2}(j)\leq Cu_{1}(i)$ for all
$i\in I$ and $j\in J$ with $Q_{i}\cap P_{j} \neq \emptyset $, then
\begin{align*}
\mathcal{D}(\mathcal{Q}, L^{p}, \ell ^{q}_{u_{1}})= \mathcal{D}(
\mathcal{P}, L^{p}, \ell ^{q}_{u_{2}})
\end{align*}
with equivalent norms.

If all sets in the coverings are connected, the conclusion holds for the
range $0 \le p,q \le \infty $.
\end{lemma}

Under suitable assumptions, the two lemmas show that an equality
\begin{equation*}
\mathcal{D}(\mathcal{Q}, L^{p_{1}}, \ell ^{q_{1}}_{u_{1}}) =
\mathcal{D}(\mathcal{P}, L^{p_{2}}, \ell ^{q_{2}}_{u_{2}})
\end{equation*}
with non-trivial exponents $(p_{1},q_{1}) \neq (2,2)$ and/or
$(p_{2},q_{2}) \neq (2,2)$ holds if and only if $\mathcal{Q}$ and
$\mathcal{P}$ are weakly equivalent, with
$(p_{1},q_{1}) = (p_{2},q_{2})$ and the involved weights are equivalent.
Note that one such equality entails
\begin{equation*}
\mathcal{D}(\mathcal{Q}, L^{p}, \ell ^{q}_{u_{1}}) = \mathcal{D}(
\mathcal{P}, L^{p}, \ell ^{q}_{u_{2}})
\end{equation*}
for \emph{all} exponents $0 \le p,q \le \infty $ and all weights
$u_{1}$, with suitably chosen $u_{2}$.

If the coverings consist of open and connected sets, weak subordinateness
implies almost subordinateness
({cf.} \cite{Voigtlaender2015PHD} Corollary~3.3.4.). Consequently, we will mainly be interested in studying necessary
and sufficient conditions for the weak subordinateness of coverings.

While the results so far compare decomposition spaces associated to different
coverings $\mathcal{Q},\mathcal{P}$ of the same set $\mathcal{O}$, the
following Theorem examines pairs of decomposition spaces associated to
coverings of different sets $\mathcal{O}$, $\mathcal{O}'$. It is a special
case of
\cite[Theorem 6.9 $\frac{1}{2}$]{VoigtlaenderEmbeddingsOfDecompositionSpaces}.

\begin{theorem}
\label{thm:different_freq_supp}
Let
$\emptyset \neq \mathcal{O},\mathcal{O}' \subset \mathbb{R}^{d}$ open.
Let $\mathcal{Q} = (Q_{i})_{i \in I}$ denote an admissible covering of
$\mathcal{O}$, $\mathcal{P} = (P_{j})_{j \in J}$ denote an admissible covering
of $\mathcal{O}'$. Assume that either
$\mathcal{O}' \cap \partial \mathcal{O} \neq \emptyset $ or
$\mathcal{O} \cap \partial \mathcal{O}' \neq \emptyset $ holds, and that
$\mathcal{O} \cap \mathcal{O}'$ is unbounded. Let
$p_{1},p_{2},q_{1},q_{2} \in (0,\infty ]$. Then
\begin{equation*}
\forall f \in C_{c}^{\infty}(\mathcal{O} \cap \mathcal{O}')~:~ \| f
\|_{D(\mathcal{Q},L^{p_{1}},\ell ^{q_{1}}_{v})} \asymp \| f \|_{D(
\mathcal{P},L^{p_{2}},\ell ^{q_{2}}_{w})}
\end{equation*}
can only hold in the trivial case, i.e., when
$(p_{1},q_{1}) = (2,2) = (p_{2},q_{2})$ and $v_{i} \asymp w_{j}$ whenever
$Q_{i} \cap P_{j} \neq \emptyset $.
\end{theorem}

Note that the assumptions on $\mathcal{O},\mathcal{O}'$ are fulfilled if
they are distinct open and dense subsets. Density and openness of
$\mathcal{O}$ imply
$\partial \mathcal{O} = \mathbb{R}^{d} \setminus \mathcal{O}$, and thus
$\mathcal{O}' \cap \partial \mathcal{O} = \emptyset $ can only happen if
$\mathcal{O}' \subsetneq \mathcal{O}$. In that case however we get
$\partial \mathcal{O}' \cap \mathcal{O} = (\mathbb{R}^{d} \setminus
\mathcal{O}') \cap \mathcal{O} \neq \emptyset $. Furthermore,
$\mathcal{O} \cap \mathcal{O}'$ is dense, hence unbounded.

\begin{remark}
\label{rem:covering_connected}
The results in \cite{VoigtlaenderEmbeddingsOfDecompositionSpaces} exhibit
an occasional subtle influence of the various assumptions on the elements
of the coverings $\mathcal{P}$ and $\mathcal{Q}$ on the comparison of associated
decomposition spaces. An instance can be witnessed in \reftext{Lemma~\ref{lem:weak_equiv_suff}}, where connectedness assumptions of the covering
sets, as well as tightness and structuredness of the cover, have an influence
on the range of summability parameters for which the conclusions of the lemma hold. The connectedness assumption will
play a more explicit role in the formulation and proof of our metric characterization
of weak equivalence in \reftext{Theorem~\ref{thm:WeakEquivQuasiIso}}.

For further remarks on these issues, we refer to \ref{rem:covering_connected_II} below.
\end{remark}

For the discussion of coorbit spaces associated to different dilation groups,
a certain class of weights will be of particular importance. Throughout
the paper, we will use $|A|$ to denote the Lebesgue measure of a Borel
set $A \subset \mathbb{R}^{d}$.

\begin{definition}
\label{defn:intrinsic_weight}
Let $\mathcal{Q} = (Q_{i})_{i \in I}$ denote a (tight) structured admissible
covering, and $\alpha \in \mathbb{R}$. The
\textit{intrinsic weight with exponent $\alpha $} is the family
$(u_{i})_{i \in I}$ defined by
\begin{equation*}
u_{i} = |Q_{i}|^{\alpha}~.
\end{equation*}
\end{definition}

Note that for $Q_{i} = T_{i}Q + b_{i}$, one obtains
$|Q_{i}| \asymp |\det (T_{i})|$. Hence the intrinsic weight is indeed
$\mathcal{Q}$-moderate, by the condition \reftext{(\ref{eqn:str_cov_norm})} on structured
coverings.

The following observation notes that intrinsic weights are robust under
weak equivalence.
%
\begin{lemma}
\label{lem:intrinsic_weight_equiv}
Let $\mathcal{Q} = (Q_{i})_{i \in I}$ and
$\mathcal{P} = (P_{j})_{j \in J}$ denote two tight structured admissible
coverings consisting of connected sets. If $\mathcal{Q}$ and
$\mathcal{P}$ are weakly equivalent, then the associated intrinsic weights
are equivalent as well, i.e., there exists a constant $C>0$, such that
one has
\begin{equation*}
\forall i \in I \forall j \in J : Q_{i} \cap P_{j} \neq \emptyset
\Rightarrow C^{-1} \le \frac{|P_{j}|}{|Q_{i}|} \le C~.
\end{equation*}
\end{lemma}

\begin{proof}
Let $x \in Q_{i} \cap P_{j}$ be given, for suitable $i \in I$ and
$j \in J$. We first note that if $\mathcal{Q}$ is weakly subordinate to
$\mathcal{P}$, and consists of connected sets, then it is almost subordinate
to $\mathcal{P}$, by
\cite[Corollary 2.13]{VoigtlaenderEmbeddingsOfDecompositionSpaces}. Hence
there exists $n \in \mathbb{N}$ such that for each $i \in I$ there exists
$j_{0} \in J$ with
\begin{equation*}
Q_{i} \subset P_{j_{0}}^{n\ast}~.
\end{equation*}
Let $C_{1}, C_{2}\ge 1$ denote the constants such that
\begin{equation*}
\forall j,\ell \in J : P_{j} \cap P_{\ell }\neq \emptyset
\Rightarrow \frac{|P_{j}|}{|P_{\ell}|} \le C_{1}~,
\end{equation*}
which results from condition \reftext{(\ref{eqn:str_cov_norm})}, and
\begin{equation*}
C_{2} = \sup _{j \in J} \sharp \{ \ell \in J : P_{j} \cap P_{\ell }
\neq \emptyset \}~.
\end{equation*}
Then $P_{j_{0}}^{n\ast}$ arises as the union of at most
$\sum _{k=0}^{n} C_{2}^{k} \le C_{2}^{n+1}$ sets, each of measure at most
$C_{1}^{n} |P_{j_{0}}|$, i.e.,
\begin{equation*}
|Q_{i}| \le |P_{j_{0}}^{n \ast}| \le |P_{j_{0}}| C_{1}^{n} C_{2}^{n+1}~.
\end{equation*}
But the fact that $x \in P_{j} \cap P_{j_{0}}^{n \ast}$ furthermore implies
that
\begin{equation*}
\frac{|P_{j_{0}}|}{|P_{j}|} \le C_{1}^{n+1}~,
\end{equation*}
whence finally
\begin{equation*}
|Q_{i}| \le |P_{j}| C_{1}^{2n+1} C_{2}^{n+1}~.
\end{equation*}
Here it is important to observe that both constants $C_{1}$ and
$C_{2}$ are independent of $i \in I$. Hence this estimate shows one direction
of the desired equivalence, and the other one follows by symmetry.
\end{proof}

\subsection{Admissible groups and induced coverings}
\label{sec2.2}

For a closed matrix group $H\leq \mathrm{GL}(\mathbb{R}^{d})$, which we
also call \textit{dilation group} in the following, we define the group
$G:=\mathbb{R}^{d}\rtimes H$, generated by dilations with elements of
$H$ and arbitrary translations, with the group law
$(x,h)\circ (y,g) := (x+hy, hg)$. We denote integration with respect to
a left Haar measure on $H$ with $\mathrm{d}h$, the associated left Haar
measure on $G$ is then given by
$d(x,h)=\mathopen{\lvert }\det h \mathclose{\rvert }^{-1}\mathrm{d}x
\mathrm{d}h$. The Lebesgue spaces on $G$ are always defined through integration
with respect to a Haar measure. The group $G$ acts on the space
$\mathrm{L}^{2}(\mathbb{R}^{d})$ through the
\textit{quasi-regular representation} $\pi $ defined by
$[\pi (x,h)f](y):=\mathopen{\lvert }\det h \mathclose{\rvert }^{-1/2} f(h^{{-1}}(y-x))$
for $f\in \mathrm{L}^{2}(\mathbb{R}^{d})$. The
\textit{generalized continuous wavelet transform (with respect to
$\psi \in \mathrm{L^{2}}(\mathbb{R}^{d})$)} of $f$ is then given as the
function
$W_{\psi }f:G \to \mathbb{C}: (x,h)\mapsto \braket{f, \pi (x,h)\psi}$. We will sometimes use the notation $W_\psi^H f$, in places where an explicit reference to the choice of dilation group $H$ is required. 
Important properties of the map $W_{\psi}:f\mapsto W_{\psi }f$ depend on
$H$ and the chosen $\psi $. If the quasi-regular representation is
\textit{square-integrable}, which means that there exists
$\psi \neq 0$ with $W_{\psi }\psi \in \mathrm{L}^{2}(G)$, and irreducible,
then we call $H$ \textit{admissible}. In this case the map
$W_{\psi}:\mathrm{L}^{2}(\mathbb{R}^{d}) \to \mathrm{L}^{2}(G)$ is a multiple
of an isometry, which gives rise to the (weak-sense) inversion formula
%
\begin{equation}
\label{eqn:waverec}
f = \frac{1}{C_{\psi}}\int _{G} W_{\psi }f(x,h) \pi (x,h) \psi
\mathrm{d}(x,h) ~,
\end{equation}
i.e., each $f \in \mathrm{L}^{2}(\mathbb{R}^{d})$ is a continuous superposition
of the wavelet system. According to results in
\cite{FuehrWaveletFramesAndAdmissibilityInHigherDImensions},
\cite{FuehrGeneralizedCalderonConditionsAndRegularOrbitSpaces}, the admissibility
of $H$ can be characterized by the \textit{dual action} defined by
$\mathbb{R}^{d} \times H \to \mathbb{R}^{d}, (h,\xi ) \mapsto h^{-T}
\xi $. In fact, $H$ is admissible iff the dual action has a single open
orbit $\mathcal{O}:=H^{{-T}}\xi _{0}\subset \mathbb{R}^{d}$ of full measure
for some $\xi _{0}\in \mathbb{R}^{d}$ and additionally the isotropy group
$H_{\xi _{0}}:=\Set{h :p_{\xi _{0}}(h)=\xi _{0}}\subset H$ is compact;
see e.g.
\cite{FuehrGeneralizedCalderonConditionsAndRegularOrbitSpaces}.

Every admissible group gives rise to an associated a covering. This is
done using the \emph{dual action} by picking a \emph{well-spread} family
in $H$, i.e. a family of elements $(h_{i})_{i\in I} \subset H$ with the
properties
\begin{enumerate}[ii)]
\item[i)] there exists a relatively compact neighborhood $U\subset H$ of the
identity such that $\bigcup _{i\in I}h_{i} U= H$ -- we say
$(h_{i})_{i\in I}$ is \textit{$U$-dense} in this case -- and
\item[ii)] there exists a neighborhood $V\subset H$ of the identity such that
$h_{i}V \cap h_{j}V = \emptyset $ for $i\neq j$ -- we say
$(h_{i})_{i\in I}$ is \textit{$V$-separated} in this case.
\end{enumerate}
The \textit{dual covering induced by $H$} is then given by the family
$\mathcal{Q}=(Q_{i})_{i\in I}$, where
$Q_{i} = p_{\xi _{0}}(h_{i} U)$ for some $\xi _{0}$ with
$H^{-T}\xi _{0}=\mathcal{O}$. It can be shown that well-spread families
always exist, and that the induced covering is indeed a tight structured
admissible covering in the sense of decomposition space theory, for which
$\mathrm{L}^{p}$-BAPUs exist, according to Theorems 4.4.6 and 4.4.13 of
\cite{Voigtlaender2015PHD}. Furthermore, there always exist induced coverings
consisting of open and connected sets, an additional feature which can
facilitate the investigations in some cases. For ease of reference, we
state this as a lemma.

\begin{lemma}[\cite{KochDoktorarbeit} Corollary 2.5.9]%
\label{cor:ExistenceInducedConnectedCovering}
Let $H$ denote an admissible dilation group, with open dual orbit
$\mathcal{O}$. Then there always exists an induced {covering} of
$\mathcal{O}$ by $H$ that is a tight structured admissible covering consisting
of (path-) connected open sets.
\end{lemma}

We call any induced covering that is a structured admissible covering consisting
of open and connected sets an
\textit{induced connected covering of $\mathcal{O}$ by }$H$. Furthermore,
two different induced coverings of the same group are always weakly equivalent,
see \cite{KochDoktorarbeit} Corollary 2.6.5.).

In Subsection \ref{subsect:CoarseEquivalenceOfOrbitAndGroup}, the next
result will be crucial.

\begin{lemma}[cf. \cite{FuehrVoigtlWavCooSpaViewAsDecSpa} Lemma 18]%
\label{lem:InducedCoveringNormCondition}
Let $(h_{i})_{i\in I}$ be a well-spread family in $H$ and let
$ K_{1}, K_{2} \subset \mathcal{O}$ be compact sets. Given $h\in H$, let
\begin{align*}
I_{h}(K_{1}, K_{2}):=
\Set{i\in I | h^{{-T}}K_{1} \cap h_{i}^{{-T}}K_{2} \neq \emptyset }.
\end{align*}
There are $C_{1} = C_{1}(K_{1},K_{2})>0$ and
$C_{2} = C_{2}(K_{1}, K_{2}, (h_{i})_{i\in I})>0$ with
$\mathopen{\lvert }I_{h}\mathclose{\rvert } \leq C_{2}$ for all
$h\in H$ and with
$\mathopen{}\mathclose{\left\lVert h^{T} h_{i}^{{-T}}\right\rVert }
\leq C_{1}$ for all $i\in I_{h}$.
\end{lemma}

\subsection{Coorbit theory and its connection to decomposition spaces}
\label{sec2.3}

Coorbit spaces are defined in terms of the decay behavior of the generalized
wavelet transform. To give a precise definition, we introduce weighted
mixed $\mathrm{L}^{p}$-spaces on $G$, denoted by
$\mathrm{L}^{p,q}_{v}(G)${.} By definition, this space is the set of functions
\begin{align*}
\left \{ f:G\to \mathbb{C} : \int _{H}\left ( \int _{\mathbb{R}^{d}}
\left | f(x,h) \right |^{p} v(x,h)^{p} \mathrm{d}x \right )^{q/p}
\frac{\mathrm{d}h}{|\det (h)|} <\infty \right \},
\end{align*}
with natural (quasi-)norm
$\Vert \cdot \Vert _{\mathrm{L}^{p,q}_{v}}$. This definition is valid for
$0< p,q <\infty $, for $p=\infty $ or $q=\infty $ the essential supremum
has to be taken at the appropriate place instead. The function
$v:G\to \mathbb{R}^{>0}$ is a measurable weight function that fulfills
the condition $v(ghk)\leq v_{0}(g)v(h)v_{0}(k)$ for some that is \textit{submultiplicative}, i.e. fulfills $v_0(gk) \le v_0(g) v_0(k)$. If this last condition on $v$ is satisfied, we call
$v$ left- and right moderate with respect to $v_{0}$. Thus, the expression
$\Vert W_{\psi }f\Vert _{\mathrm{L}^{p,q}_{v}}$ can be read as a measure
of wavelet coefficient decay of $f$. We consider weights which only depend
on $H$. The coorbit space
$\mathrm{Co}\left (\mathrm{L}^{p,q}_{v}(\mathbb{R}^{d}\rtimes H)
\right )$ is then defined as the space
%
\begin{align}
\label{def:Coorbit}
\left \{ f\in \mathcal{(H}^{1}_{w})^{\neg }: W_{\psi }f \in W(
\mathrm{L}^{p,q}_{v}(\mathbb{R}^d \rtimes H))\right \}
\end{align}
for a suitable wavelet $\psi $ fulfilling various technical conditions,
and some control weight $w$ associated to $v$. The space
$(\mathcal{H}^{1}_{w})^{\neg}$ denotes the space of antilinear functionals
on
$ \mathcal{H}^{1}_{w} :=\left \{ f\in \mathrm{L}^{2}(\mathbb{R}^{d}): W_{
\psi }f \in \mathrm{L}^{1}_{w}(G)\right \} $ and $W(Y)$ for a function
space $Y$ on $G$ denotes the Wiener amalgam space defined by
$ W_{Q}(Y):=\{f\in \mathrm{L}^{\infty}_{\text{loc}}(G) | M_{Q}f\in Y\} $
with quasi-norm $\|f\|_{W_{Q}(Y)}:=\|M_{Q}f\|_{Y}$ for
$f\in W_{Q}(Y)$, where the \textit{maximal function} $M_{Q}f$ for some suitable
unit neighborhood $Q\subset G$ is
$ M_{Q}f:G\to [0,\infty ],\ x\mapsto \operatorname{ess\ sup}_{y\in xQ}|f(y)|$.

The appearance of the Wiener amalgam space in \reftext{(\ref{def:Coorbit})} is necessary
to guarantee consistently defined quasi-Banach spaces in the case
$\{p,q\}\cap (0,1)\neq \emptyset $, see
\cite{Rauhut2007CoorbSpacTheoForQuasiBanSpa} and
\cite{Voigtlaender2015PHD}. In the classical coorbit theory for Banach
spaces, which was developed in \cite{FeiGroeBanSpaRelToIntGrpRepI,FeiGroeBanSpaRelToIntGrpRepII}, the Wiener amalgam space is replaced
by $\mathrm{L}^{p,q}_{v}(G)$ and this change leads to the same space for
$p,q\geq 1$, see \cite{Rauhut2007CoorbSpacTheoForQuasiBanSpa}.

Many useful properties of these spaces are known and hold in the quasi-Banach
space case as well as in the Banach space case. The most prominent examples
of coorbit spaces associated to generalized wavelet transforms are the
homogeneous Besov spaces and the modulation spaces. However, each shearlet
group, a class of groups we introduce in the next subsection, gives rise
to its own scale of coorbit spaces, as well; see
\cite{kutyniok2012shearlets},
\cite{DahlkeHaeuTesCooSpaTheForTheToeSheTra} and
\cite{FuehCooSpaAndWavCoeDecOveGenDilGro}.

\begin{remark}
\label{rem:coorbit_reservoir}
The use of the reservoir space $(\mathcal{H}^{1}_{w})^{\neg}$ is a major
obstacle for the comparison of coorbit spaces associated to different dilation
groups. From a purely set-theoretic point of view, different dilation groups
necessarily result in disjoint coorbit spaces. However, this distinction
is somewhat beside the point, and it can be remedied by the following observations;
see also \cite[Corollary 4.4]{FeiGroeBanSpaRelToIntGrpRepI}. By the Fischer-Riesz
theorem, the inclusion
$\mathcal{H}^{1}_{w} \subset L^{2}(\mathbb{R}^{d})$ provides a canonical
embedding
$L^{2}(\mathbb{R}^{d}) \subset (\mathcal{H}^{1}_{w})^{\neg}$, which allows
to identify $L^{2}(\mathbb{R}^{d})$ with a subspace of
$(\mathcal{H}^{1}_{w})^{\neg}$. Now the fact that the wavelet transform
is $L^{2}$-isometric (and the associated inversion formula) implies that
the range of this identification map is precisely given by
$Co(L^{2}(\mathbb{R}^{d} \rtimes H))$.

We next observe that
$Co(L^{p}(\mathbb{R}^{d} \rtimes H)) \subset Co(L^{2}(\mathbb{R}^{d}
\rtimes H))$ for $1 \le p \le 2$. This follows from the fact that the discrete
sequence space $Y_{d}$, associated to
$Y =L^{p}(\mathbb{R}^{d} \times H)$ according to Definition 3.4 of
\cite{FeiGroeBanSpaRelToIntGrpRepI}, turns out to be
$Y_{d} = \ell ^{p} \subset \ell ^{2}$ (see
\cite[Lemma 3.5(e)]{FeiGroeBanSpaRelToIntGrpRepI}). Furthermore,
$Co(Y) \subset Co(Z)$ holds iff $Y_{d} \subset Z_{d}$, by
\cite[Theorem 8.4]{FeiGroeBanSpaRelToIntGrpRepII}, hence
$Co(L^{p}) \subset Co(L^{2})$.

Combining this embedding with the identification of
$Co(L^{2}(\mathbb{R}^{d} \rtimes H))$ with $L^{2}(\mathbb{R}^{d})$, we find
the following canonical identification for $1 \le p \le 2$, valid for a
suitably chosen analyzing vector $\psi $:
\begin{equation*}
Co(L^{p}(\mathbb{R}^{d} \rtimes H)) = \{ f \in L^{2}(\mathbb{R}^{d}) : W_{
\psi}^{H} f \in L^{p}(\mathbb{R}^{d} \rtimes H) \}~,
\end{equation*}
with the obvious norm. But now the new reservoir space
$L^{2}(\mathbb{R}^{d})$ is independent of $H$.

We note that this identification extends to all weighted
$L^{p,q}_{v}(\mathbb{R}^{d} \rtimes H)$ for which the associated sequence
space fulfills $\ell ^{p,q}_{v} \subset \ell ^{2}$, in particular for
$1 \le p,q \le 2$ and constant weight $v$.

An alternative approach to the just-mentioned embeddings of coorbit spaces,
pointed out to us by the referee, is obtained by observing the following
well-known inclusion relations between Wiener amalgam spaces, valid for $0<p<2$,
\begin{equation*}
W(L^{p}) \subset W(L^{2}) \subset L^{2}~,
\end{equation*}
which via equation \reftext{(\ref{def:Coorbit})} translates to the inclusions of
the associated coorbit spaces, and has the advantage of directly generalizing
to $p<1$.
\end{remark}

The next definition will be useful for the transfer of weights from the
coorbit to the decomposition space setting.

\begin{definition}[\cite{Voigtlaender2015PHD} Definition 4.5.3.]%
\label{def:DecompositionWeight}
For $q\in (0,\infty ]$ and a weight $m: H \to (0,\infty )$, we define the
weight
$ m^{(q)}:H\to (0,\infty ),\ h\mapsto \mathopen{\lvert }\det (h)
\mathclose{\rvert }^{\frac{1}{2} - \frac{1}{q}} m(h) $. Here, we set
$\frac{1}{\infty}:=0$.
\end{definition}

The connection between coorbit spaces and decomposition spaces is given
by the next theorem. For the Banach space case, we also refer to
\cite{FuehrVoigtlWavCooSpaViewAsDecSpa}. A recent extension beyond the
irreducible setting can be found in \cite{fuehrvelthoven2020coorbit}.

\begin{theorem}[\cite{Voigtlaender2015PHD} Theorem 4.6.3]%
\label{thm:FourierIsoCoorbitDecSpaces}
Let $\mathcal{Q}$ be a covering of the dual orbit $\mathcal{O}$ induced
by $H$, $0<p,q\leq \infty $ and $u=(u_{i})_{i\in I}$ a suitable weight,
then the Fourier transform
$ \mathcal{F}: \mathrm{Co}\left (\mathrm{L}^{p,q}_{v}(\mathbb{R}^{d}
\rtimes H)\right ) \to \mathcal{D}(\mathcal{Q}, \mathrm{L}^{p}, \ell ^{q}_{u})
$ is an isomorphism of (quasi-) Banach spaces. The weight
$(u_{i})_{i\in I}$ can be chosen as $u_{i}:=v^{(q)}(h_{i})$, where
$(h_{i})_{i\in I}$ is the well-spread family used in the construction of
$\mathcal{Q}$. We call such a weight a $\mathcal{Q}$-discretization
of $v$.
\end{theorem}

\begin{remark}
\label{rem:ind_weight_intr}
In the following, we will mostly concentrate on constant weights, i.e.
on the study of coorbit spaces of the type $Co(L^{p,q}(G))$ corresponding
to $v\equiv 1$. This has the important consequence that the
$\mathcal{Q}$-discretization $(u_{i})_{i \in I}$ obtained from a dual covering
$\mathcal{Q} = (Q_{i})_{i \in I} = (h_{i}^{-T} Q)_{i \in I}$ fulfills
\begin{equation*}
u_{i} = |{\mathrm{det}}(h_{i})|^{\frac{1}{2}-\frac{1}{q}} = |Q|^{
\frac{1}{2}-\frac{1}{q}} |Q_{i}|^{\frac{1}{q}-\frac{1}{2}} \asymp |Q_{i}|^{
\frac{1}{q}-\frac{1}{2}} ~.
\end{equation*}
In other words, the induced weight $(u_{i})_{i \in I}$ is intrinsic with
exponent $\alpha =\frac{1}{q}-\frac{1}{2}$.
\end{remark}

\begin{remark}
\label{rem:coorbit_decspace_four}
Since the domain of the Fourier transform in the previous theorem consists
of spaces that depend on the choice of dilation groups, some clarification
is in order, see
\cite[Corollary 10]{FuehrVoigtlWavCooSpaViewAsDecSpa} for more details:
Given $f \in (\mathcal{H}^{1}_{w})^{\neg}$, we let
$\mathcal{F}(f) \in \mathcal{D}'(\mathcal{O})$ be defined as the map
$g \mapsto f (\mathcal{F}^{-1} \overline{g})$. Here $\mathcal{F}$ denotes
the Fourier transform defined on the Schwartz functions, and the definition
makes use of
$\mathcal{D}(\mathcal{O}) \subset \mathcal{S}(\mathcal{O})$, as well as
$\mathcal{F}^{-1}(\mathcal{D}(\mathcal{O})) \subset \mathcal{H}^{1}_{w}$ \cite[Theorem 2.1]{FuehCooSpaAndWavCoeDecOveGenDilGro}. 

By the Fischer-Riesz theorem, the inclusion
$\mathcal{H}^{1}_{w} \subset L^{2}(\mathbb{R}^{d})$ provides a canonical
embedding
$L^{2}(\mathbb{R}^{d}) \subset (\mathcal{H}^{1}_{w})^{\neg}$. Hence for
elements of $L^{2}(\mathbb{R}^{d})$, one can rewrite the formula for
$\mathcal{F}(f)$ as
\begin{equation*}
\mathcal{F}(f) (g) = \langle \widehat{f}, \overline{g} \rangle _{L^{2}(
\mathbb{R}^{d})}~,
\end{equation*}
where $\widehat{f}$ denotes the Plancherel transform of $f$. In particular,
for all coorbit spaces $Co(L^{p}(\mathbb{R}^{d} \rtimes H))$ having a canonical embedding
into $ L^{2}(\mathbb{R}^{d})$ according to \reftext{Remark~\ref{rem:coorbit_reservoir}}, the Fourier transform in the theorem coincides
with the Plancherel transform, combined with the obvious identification
of the $L^{2}$-function $\widehat{f}$ with the distribution on
$\mathcal{D}(\mathcal{O})$ obtained by integration against
$\widehat{f}$. In this way, the Fourier transform $\mathcal{F}$ on
$Co(L^{p}(\mathbb{R}^{d} \rtimes H))$, as defined in \reftext{Theorem~\ref{thm:FourierIsoCoorbitDecSpaces}}, also becomes independent of the choice
of dilation group, whenever we canonically identify the coorbit space with a subspace
of $L^{2}(\mathbb{R}^{d})$.
\end{remark}

We next formalize the property that two admissible dilation groups have
the same coorbit spaces. We already pointed out that a literal interpretation
of this property is not generally available, at least not for all possible
choices of coorbit space norms.

\begin{definition}
\label{defn:coorbit_equivalent}
Let $H_{1}, H_{2} \le GL(\mathbb{R}^{d})$ denote admissible matrix groups.
We call $H_{1},H_{2}$ \textit{coorbit equivalent} if for all
$0 < p,q \le \infty $ and for all $f \in L^{2}(\mathbb{R}^{d})$ we have
\begin{equation*}
\| f \|_{ Co(L^{p,q}(\mathbb{R}^{d} \rtimes H_{1}))} \asymp \| f \|_{Co(L^{p,q}(
\mathbb{R}^{d} \rtimes H_{2}))}~.
\end{equation*}
Here the norm equivalence is understood in the generalized sense that one
side is infinite iff the other side is.
\end{definition}

\begin{remark}
\label{rem2.17}
An example of distinct groups that are coorbit equivalent can be found
in \cite{FuehrVoigtlWavCooSpaViewAsDecSpa}, Section 9: If
$H = \mathbb{R}^{+} \times SO(d)$, for $d>1$, and
$C \in GL(\mathbb{R}^{d})$ is arbitrary, then $H$ and $C^{-1} H C$ are
coorbit equivalent, but typically distinct.
\end{remark}

Now the decomposition space characterization from \reftext{Theorem~\ref{thm:FourierIsoCoorbitDecSpaces}}, together with the results from the
previous subsection characterizing when admissible coverings lead to identical
scales of decomposition spaces, provides a rather stringent characterization
of coorbit equivalence.

\begin{theorem}
\label{thm:coorbit_equiv_dual_orbits}
Let $H_{1}, H_{2} \le GL(\mathbb{R}^{d})$ denote admissible matrix groups,
and let $\mathcal{O}_{1},\mathcal{O}_{2}$ denote the associated open dual
orbits. Then the following are equivalent:
\begin{enumerate}
\item[(a)] $H_{1}$ and $H_{2}$ are coorbit equivalent.
\item[(b)] For all $1 \le p,q \le 2$:
$Co(L^{p,q}(\mathbb{R}^{d} \rtimes H_{1})) = Co(L^{p,q}(\mathbb{R}^{d}
\rtimes H_{2}))$, as subspaces of $L^{2}(\mathbb{R}^{d})$.
\item[(c)] There exists $1 \le p,q \le 2$ with $(p,q) \neq (2,2)$, such
that
$Co(L^{p,q}(\mathbb{R}^{d} \rtimes H_{1})) = Co(L^{p,q}(\mathbb{R}^{d}
\rtimes H_{2}))$, as subspaces of $L^{2}(\mathbb{R}^{d})$.
\item[(d)] $\mathcal{O}_{1} = \mathcal{O}_{2}$, and the coverings induced
by $H_{1}$ and $H_{2}$ on the common open orbit are weakly equivalent.
\end{enumerate}
\end{theorem}

\begin{proof}
The implication $(a) \Rightarrow (b)$ is clear by the nature of the canonical
embedding of the coorbit spaces into $L^{2}(\mathbb{R}^{d})$ detailed in
\reftext{Remark~\ref{rem:coorbit_reservoir}}. $(b) \Rightarrow (c)$ is obvious.

For the direction $(c) \Rightarrow (d)$, recall that
$\mathcal{O}_{1},\mathcal{O}_{2}$ are both open and dense. Hence, following
the observations after \reftext{Theorem~\ref{thm:different_freq_supp}}, if these
sets are distinct, then either
$\mathcal{O}_{1} \cap \partial \mathcal{O}_{2} \neq \emptyset $ or
$\mathcal{O}_{2} \cap \partial \mathcal{O}_{1} \neq \emptyset $ follows,
and in addition $\mathcal{O}_{1} \cap \mathcal{O}_{2}$ is unbounded. Moreover,
the space
$\mathcal{F}^{-1}(C_{c}^{\infty}(\mathcal{O}_{1} \cap \mathcal{O}_{2}))$
is a subspace of $L^{2}(\mathbb{R}^{d})$, and the assumption on the coorbit
spaces combined with \reftext{Theorem~\ref{thm:FourierIsoCoorbitDecSpaces}}, yields
that the decomposition space norms
$\| \cdot \|_{\mathcal{D}(\mathcal{Q}, \mathrm{L}^{p}, \ell ^{q}_{v})}$
and
$\| \cdot \|_{\mathcal{D}(\mathcal{P}, \mathrm{L}^{p}, \ell ^{q}_{u})}$
are equivalent on
$C_{c}^{\infty}(\mathcal{O}_{1} \cap \mathcal{O}_{2})$, whenever
$\mathcal{Q}$ and $\mathcal{P}$ are coverings induced by
$H_{1}, H_{2}$, respectively. Here $u,v$ are the intrinsic weights of exponent
$\alpha =\frac{1}{q}-\frac{1}{2}$ associated to the coverings, see \reftext{Remark~\ref{rem:ind_weight_intr}}. Now the fact that $(p,q) \neq (2,2)$ implies
via \reftext{Theorem~\ref{thm:different_freq_supp}} that
$\mathcal{O}_{1} = \mathcal{O}_{2}$.

We observed in \reftext{Remark~\ref{rem:coorbit_decspace_four}} that the
$L^{2}$-Fourier transform $\mathcal{F}$ induces isomorphisms
\begin{equation*}
Co(L^{p,q}(\mathbb{R}^{d} \rtimes H_{1})) \to \mathcal{D}(\mathcal{P},
L^{p}, \ell ^{q}_{u}) ~,~ Co(L^{p,q}(\mathbb{R}^{d} \rtimes H_{2}))
\to \mathcal{D}(\mathcal{P}, L^{p}, \ell ^{q}_{v})
\end{equation*}
and by assumption, the coorbit spaces coincide. But then
\begin{equation*}
\mathcal{D}(\mathcal{Q}, L^{p}, \ell ^{q}_{u}) = \mathcal{D}(
\mathcal{P}, L^{p}, \ell ^{q}_{v}) ~.
\end{equation*}
Now \reftext{Lemma~\ref{cor:WeakEquivalenceOfInducedCoverings}} implies weak equivalence
of the coverings.

Finally, $(d) \Rightarrow (a)$ follows combining \reftext{Theorem~\ref{thm:FourierIsoCoorbitDecSpaces}} and \reftext{Lemma~\ref{lem:weak_equiv_suff}}. Observe also that the induced weights are equivalent
by \reftext{Remark~\ref{rem:ind_weight_intr}} and \reftext{Lemma~\ref{lem:intrinsic_weight_equiv}}.
\end{proof}

\begin{remark}
\label{rem2.19}
The unweighted case, which is in the focus of the theorem, covers all spaces
of the type $L^{p}(\mathbb{R}^{d} \rtimes H)$, for $1\le p <2$, which are
of special importance for nonlinear approximation. Considering general,
non-constant weights on the dilation groups requires solving a technical
issue, that is currently not well-understood. In this general setting,
the question becomes whether any choice of weight $v_{1}$ given on one
group $H_{1}$ can be matched by an appropriate weight $v_{2}$ on
$H_{2}$, such that the resulting weights on the induced coverings are equivalent.
In principle it seems possible, using cross-sections, to construct
$v_{2}$ from $v_{1}$. However, it is currently unclear whether, starting
from a weight $v_{1}$ that is moderate with respect to a suitable sub-multiplicative
function on $H_{1}$, the same will automatically hold true for the newly
constructed weight $v_{2}$ on $H_{2}$.
\end{remark}

\section{A metric characterization of weak equivalence}
\label{sect:metric_reformulation}

This section contains the central observation of this paper, namely that
the question of weak equivalence of coverings, and thus of equality of
the scales of decomposition spaces associated to the coverings, can be
reformulated using the language of coarse geometry.

\subsection{Elementary notions and results from coarse geometry}
\label{subsect:elem_coarse}

In this section we collect the basic results from coarse geometry that
are needed for the following. Since the terminology employed in coarse
geometry is not entirely unified, and for reasons of self-containedness,
we include proofs of some basic results for which we could not find a handy
reference.

\begin{remark}%
\label{rem:OnlyCoarseStructureOfMetric}
In the following, we define different metrics on subsets of
$\mathbb{R}^{d}$ and on matrix groups, but we will only be concerned with
the coarse properties of these metrics. In particular, the terms
\textit{open}, \textit{closed}, \textit{connected component} and
\textit{the continuity of maps} between these spaces have to be understood
with respect to the standard topology on the respective sets and not with
respect to the topology that may be induced by these different metrics.
\end{remark}

The pivotal role continuous maps play in the study of topological spaces
is taken over by coarse maps in the study of large scale properties of
metric spaces.

\begin{definition}[cf. \cite{Roe2003Lectures} Definition 1.8]
\label{defn3.2}
Let $(X, d_{X})$ and $(Y, d_{Y})$ be metric spaces and let
$f:X\to Y$ be a map.
\begin{itemize}
\item We call $f$ \textit{metrically proper} if the inverse image under
$f$ of every bounded subset in $Y$ is bounded in $X$, that is
\begin{align*}
&\sup \left \{d_{Y}(y,y') :\ y,y'\in A \right \}< \infty
\\
& \quad \Rightarrow \sup \left \{d_{X}(x,x') :\ x,x'\in f^{{-1}}(A)
\right \}< \infty
\end{align*}
for all sets $A\subset Y$.
\item We call $f$ \textit{uniformly bornologous} if for every $R>0$ there
exists $S>0$ such that
\begin{align*}
d_{X}(x,x')<R\quad \Longrightarrow \quad d_{Y}\left (f(x),f(x')
\right )<S
\end{align*}
for all $x,x'$ in $X$.
\item A \textit{coarse map} from $(X, d_{X})$ to $(Y, d_{Y})$ is a metrically
proper and uniformly bornologous map.
\end{itemize}
\end{definition}

A simple argument shows that the composition of coarse maps is coarse as
well. In order to define when we want to consider coarse structures
on different spaces as essentially the same, we introduce the notion of closeness
of maps.

\begin{definition}[cf. \cite{NowakYu2012} Definition 1.3.2]
\label{defn3.3}
Let $(X, d_{X})$ and $(Y, d_{Y})$ be metric spaces and $f:X\to Y$ and
$g:X \to Y$ be maps. We say $f$ and $g$ are \emph{close} if
\begin{align*}
\sup \left \{d_{Y}\left (f(x), g(x)\right ):\ x\in X \right \} <
\infty .
\end{align*}
\end{definition}

\begin{definition}[cf. \cite{Roe2003Lectures} Page 6]
\label{defn3.4}
Two metric spaces $(X, d_{X})$ and $(Y, d_{Y})$ are
\emph{coarsely equivalent} if there exist two coarse maps
\begin{equation*}
f:X\to Y\quad \text{and}\quad g:Y\to X%
\end{equation*}
such that $f \circ g$ is close to $\mathrm{id}_{Y}$ and $g \circ f$ is
close to $\mathrm{id}_{X}$. $f$ is then called a
\emph{metric coarse equivalence}.
\end{definition}

As the name suggests, this relation is in fact an equivalence relation.
An important class of coarse maps are quasi-isometries.

\begin{definition}[cf. \cite{Roe2003Lectures} Remark 1.9]
\label{defn3.5}
Let $(X, d_{X})$ and $(Y, d_{Y})$ be metric spaces and $f:X\to Y$ be a
map. We call $f$ \textit{large-scale Lipschitz} if there are $L,C>0$ such
that
$d_{Y}\left (f(x), f(x')\right ) \leq Ld_{X}\left (x,x'\right ) + C$ for
all $x,x'\in X$.
\end{definition}

\begin{definition}[cf. \cite{NowakYu2012} Definition 1.3.4]%
\label{def:quasi-isometry}
Let $(X, d_{X})$ and $(Y, d_{Y})$ be metric spaces. A map $f:X\to Y$ is
a \textit{quasi-isometry} if the following conditions are satisfied:
\begin{enumerate}[ii)]
\item[i)] The map $f$ is a \textit{quasi-isometric embedding}. This means there
exist constants $L,C>0$ such that
\begin{align*}
L^{-1}d_{X}(x,x')-C \leq d_{Y}\left (f(x),f(x')\right ) \leq L d_{X}(x,x')
+ C
\end{align*}
for all $x,x'\in X$.
\item[ii)] The map $f$ is \textit{coarsely surjective}. This means that there
exists $K>0$ such that for every $y\in Y$ exists an $x\in X$ with
$d_{Y}\left (f(x), y\right )\leq K$.
\end{enumerate}
\end{definition}

\begin{remark}
\label{rem3.7}
It is easy to see that a surjective map is coarsely surjective. In the
concrete cases we consider, this will be the standard justification for
this property.
\end{remark}

As for coarse maps, it is easy to see that the composition of quasi-isometric
maps is again quasi-isometric. The following lemma has a straightforward
proof.

\begin{lemma}%
\label{lem:QuasiIsoCaorse}
Every quasi-isometric embedding is coarse.
\end{lemma}

\begin{definition}%
\label{def:CoarseRightInverse}
Let $(X, d_{X})$ and $(Y, d_{Y})$ be metric spaces and $f:X\to Y$. We call
a map $h:Y\to X$ such that $f\circ h$ is close to $\mathrm{id}_{Y}$ a
\textit{coarse right inverse of $f$}.
\end{definition}

\begin{example}
\label{exmp3.10}
The map
$f:(\mathbb{R}, |\cdot|)\to (\mathbb{R}^{2}, |\cdot|_{1}),\ x\mapsto (x,0)^{T}$ illustrates
that not every coarse map has a coarse right inverse.
\end{example}

Quasi-isometries always have coarse right inverses. We include a proof,
for the reasons mentioned above.

\begin{lemma}%
\label{lem:QuasiIsoHasCoaInv}
Every quasi-isometry has a coarse right inverse. Furthermore, any coarse
right inverse of a quasi-isometry is a quasi-isometry.
\end{lemma}

\begin{proof}
Let $(X, d_{X})$ and $(Y, d_{Y})$ be metric spaces and $f:X\to Y$ be a
quasi-isometry. This entails the existence of some $K>0$ with the property
that for every $y\in Y$ exists an $x\in X$ with
$d_{Y}\left (f(x), y\right )\leq K$. Consequently, for every
$y\in Y$, the set
\begin{equation*}
A_{y}:=\left \{x\in X:\ d_{Y}\left (f(x), y\right ) \leq K \right \}%
\end{equation*}
is nonempty. The axiom of choice implies the existence of at least one
coarse right inverse of $f$ because $\prod _{y\in Y}A_{y}$ is nonempty.

Since $f$ is a quasi-isometry, there exist $L,C>0$ with
\begin{align*}
L^{-1}d_{X}(x,x')-C \leq d_{Y}\left (f(x),f(x')\right ) \leq L d_{X}(x,x')
+ C
\end{align*}
for all $x,x'\in X$. Moreover, let $h:Y\to X$ be any coarse right inverse
of $f$ such that
$\sup _{y\in Y}d_{Y}\left ((f\circ h)(y),y)\right ) \leq M$ for some
$M>0$. By application of the triangle inequality twice, it follows that
\begin{align*}
L^{{-1}}d_{X}\left (h(y), h(y')\right ) - C &\leq d_{Y}\left ((f
\circ h)(y),(f\circ h)(y')\right )
\\
&\leq d_{Y}\left ((f\circ h)(y), y\right ) + d_{Y}\left (y', y\right )
\\
& \qquad + d_{Y}\left (y', (f\circ h)(y')\right )
\\
& \leq 2M + d_{Y}\left (y', y\right )
\end{align*}
for all $y, y' \in Y$. This inequality implies
\begin{align*}
d_{X}\left (h(y), h(y')\right ) \leq Ld_{Y}\left (y', y\right ) + 2LM +
CL.
\end{align*}
We continue by using the reverse triangle inequality twice to get
\begin{align*}
Ld_{X}\left (h(y), h(y')\right ) + C &\geq d_{Y}\left ((f\circ h)(y),(f
\circ h)(y')\right )
\\
&\geq d_{Y}\left (y,(f\circ h)(y')\right ) - d_{Y}\left ((f\circ h)(y),
y\right )
\\
&\geq d_{Y}\left (y,y'\right ) - d_{Y}\left ((f\circ h)(y'), y'
\right )
\\
&\qquad - d_{Y}\left ((f\circ h)(y), y\right )
\\
&\geq d_{Y}\left (y,y'\right ) - 2M.
\end{align*}
This inequality implies
\begin{align*}
d_{X}\left (h(y), h(y')\right ) \geq L^{{-1}}d_{Y}\left (y', y\right )
- 2L^{{-1}}M - CL^{{-1}}
\end{align*}
for all $y,y' \in Y$. It follows that $h$ is a quasi-isometric embedding.

In order to show that $h$ is a quasi-isometry, let $x\in X$ be arbitrary
and choose $y:=f(x)$. We conclude
\begin{align*}
d_{X}\left (h (f(x)),x \right ) &\leq Ld_{Y}\left (f\circ (h\circ f)(x),f(x)
\right ) + CL
\\
& = Ld_{Y}\left ((f\circ h)(f(x)),f(x) \right ) + CL
\\
& = Ld_{Y}\left ((f\circ h)(y),y \right ) + CL
\\
& \leq LM + CL
\end{align*}
and the inequality
%
\begin{align}
\label{eq:CoarseRightInverse}
d_{X}\left (h (f(x)),x \right ) \leq LM+CL
\end{align}
shows that $h$ is coarsely surjective. In summary, this shows that $h$ is a quasi-isometry.
\end{proof}

\begin{remark}%
\label{rem:CoarseRightInverse}
The inequality \reftext{(\ref{eq:CoarseRightInverse})} also shows that if
$f:X\to Y$ is a quasi-isometry and $h:Y\to X$ is any coarse right inverse
of $f$, then $f$ is a coarse right inverse of $h$. If we define a coarse
left inverse analogously, then this shows that every coarse right inverse
of a quasi-isometry is also a coarse left inverse of this quasi-isometry.
\end{remark}

\begin{theorem}%
\label{thm:QuasiIsoCoarseEqiuv}
Let $(X, d_{X})$ and $(Y, d_{Y})$ be metric spaces. If there exists a quasi-isometry
$f:X\to Y$, then $X$ and $Y$ are coarsely equivalent.
\end{theorem}

\begin{proof}
According to \reftext{Lemma~\ref{lem:QuasiIsoHasCoaInv}}, $f$ has a coarse right
inverse $h$, which is also a quasi-isometry. By \reftext{Lemma~\ref{lem:QuasiIsoCaorse}}, $f$ and $h$ are coarse maps. Since $h$ is a coarse
right inverse of $f$, the map $f\circ h$ is close to
$\mathrm{id}_{Y}$. From \reftext{Remark~\ref{rem:CoarseRightInverse}} follows that
$h\circ f$ is close to $\mathrm{id}_{X}$. This shows that $f$ and
$h$ induce a coarse equivalence of $X$ and $Y$.
\end{proof}

The following definition allows to identify a general setting in which
the terms ``coarse equivalence'' and ``quasi-isometry'' are synonymous.
See \cite[Definition 3.B.1]{Cornulier_2016}.
%
\begin{definition}
\label{defn:ls_geodesic}
Let $(X,d)$ denote a metric space. For $c>0$, $X$ is called $c$-large-scale
geodesic if there exist constants $a>0, b \ge 0$ such that for every pair
of points $x,x'$ there exists a finite sequence
$x_{0} = x,x_{1},\ldots , x_{n} = x'$ satisfying
$d(x_{i-1},x_{i}) \le c$ and $n \le a d(x,x') + b$. $(X,d)$ is called large-scale
geodesic if it is $c$-large-scale geodesic for some $c>0$.
\end{definition}

The following result is \cite[Proposition 3.B.9]{Cornulier_2016}.
%
\begin{proposition}
\label{prop:ce_eq_qi}
Let $(X,d)$ and $(Y,d')$ denote large-scale geodesic spaces, and
$f: X \to Y$. Then $f$ is a coarse equivalence if and only if it is a quasi-isometry.
\end{proposition}
%

\subsection{Metrics induced by coverings}
\label{sec3.2}

We start with a preparation for the definition of the metric associated
to a covering.

\begin{definition}
\label{defn3.16}
Let $\mathcal{O}\subset \mathbb{R}^{d}$ be open and let
$\mathcal{Q}= (Q_{i})_{i\in I}$ be a covering of $\mathcal{O}$. For
$x,y\in \mathcal{O}$, we say $x$ and $y$ are connected by a
\textit{$\mathcal{Q}$-chain (of length m)} if there exist
$Q_{1}, \ldots , Q_{m} \in \mathcal{Q}$ such that $x\in Q_{1}$,
$y\in Q_{m}$ and $Q_{k}\cap Q_{k+1}\neq \emptyset $ for all
$k\in \{1,\ldots , m-1\}$.
\end{definition}

We next define the metric, which is called
\textit{$\mathcal{Q}$-chain distance} in
\cite[Definition 3.4]{FeiGroeBanSpaRelToIntGrpRepI}.

\begin{definition}
\label{defn3.17}
Let $\mathcal{O}\subset \mathbb{R}^{d}$ be open and let
$\mathcal{Q}= (Q_{i})_{i\in I}$ be a covering of $\mathcal{O}$. Define
the map
$d_{\mathcal{Q}}:\mathcal{O}\times \mathcal{O}\to \mathbb{N}_{0}\cup
\{\infty \}$ by
\begin{align*}
d_{\mathcal{Q}}(x,y)=
\begin{cases}
\inf
\Set{m\in \mathbb{N}|
\begin{array}{l}
x,y \text{ are connected by a}\\
\mathcal{Q}\text{-chain of length } m
\end{array}
}, & x\neq y
\\
0, &x=y,
\end{cases}
\end{align*}
where we set $\inf \emptyset =\infty $.
\end{definition}

The above map in fact defines a metric (we allow the value $\infty $ for
a metric) on $\mathcal{O}$, without any further restrictions on the covering.

\begin{remark}
\label{rem3.18}
Let $\mathcal{O}\subset \mathbb{R}^{d}$ be open and let
$\mathcal{Q}= (Q_{i})_{i\in I}$ be a covering of $\mathcal{O}$. One could
also be tempted to define a function by considering the neighbors of a
set (\textbf{recall \reftext{Definition~\ref{def:SetOfNeighbors}}})
\begin{align*}
\widetilde{d}:&\mathcal{O}\times \mathcal{O}\to \mathbb{N}_{0}\cup \{
\infty \},
\\
&(x,y)\mapsto
\begin{cases}
\inf
\Set{n\in \mathbb{N}| \exists Q\in \mathcal{Q}\text{ with } x,y \in Q^{n \ast}},
& x\neq y
\\
0 & x=y.
\end{cases}
\end{align*}
In general, this map does not satisfy the triangle inequality and is therefore
not a metric on $\mathcal{O}$. It suffices to consider a uniform covering
of $\mathcal{O}=\mathbb{R}$ by intervals of length $2$ as depicted in the
diagram below. Here, the intervals are represented by rectangles. For the
points $x,y,z \in \mathbb{R}$, we have $\widetilde{d}(x,y)=1$,
$\widetilde{d}(y,z)=1$ but $\widetilde{d}(x,z)=3$.

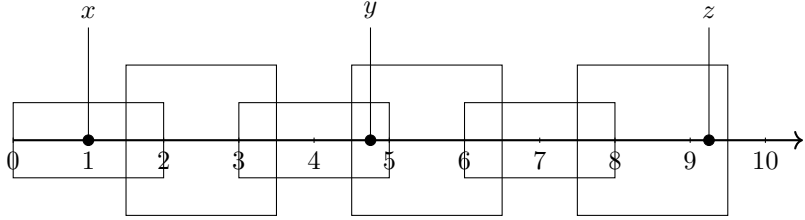
\begin{center}
\begin{tikzpicture}
\draw [thick,->] (0,0) -- (10.5,0);
\foreach \x in {0,1,2,3,4,5,6,7,8,9,10}
    \draw (\x cm,1pt) -- (\x cm,-1pt) node[anchor=north] {$\x$};

    \draw (0,-0.5) rectangle (2,0.5);
    \draw (1.5,-1) rectangle (3.5,1);
    \draw (3,-0.5) rectangle (5, 0.5);
    \draw (4.5,-1) rectangle (6.5,1);
    \draw (6,-0.5) rectangle (8, 0.5);
    \draw (7.5,-1) rectangle (9.5,1);

    \filldraw (1,0) circle (2pt);
    \filldraw (4.75,0) circle (2pt);
    \filldraw (9.25,0) circle (2pt);

    \draw (1,0) -- (1,1.5);
    \draw (4.75,0) -- (4.75,1.5);
    \draw (9.25,0) -- (9.25,1.5);

    \draw (1,1.5) node[above] {$x$};
    \draw (4.75,1.5) node[above] {$y$};
    \draw (9.25,1.5) node[above] {$z$};
\end{tikzpicture}
\end{center}
\end{remark}

\subsection{Characterizing weak equivalence}
\label{sec3.3}

We will only consider coverings consisting of open connected sets. Here
the property $d(x,y) = \infty $ can be related to the partition of
$\mathcal{O}$ into connected components.

\begin{lemma}
\label{lem3.19}
Let $\mathcal{O}\subset \mathbb{R}^{d}$ be open and let
$\mathcal{Q}= (Q_{i})_{i\in I}$ be a covering of $\mathcal{O}$ comprised
of open connected subsets of $\mathcal{O}$. Then
$d_{\mathcal{Q}}(x,y) < \infty $ if and only if $x,y$ lie in the same connected
component of $\mathcal{O}$.
\end{lemma}

\begin{proof}
We consider the equivalence relation
$x \sim y :\Leftrightarrow d(x,y) < \infty $. Fix $x$, and let
$y \in \mathcal{O}$ with $ x \sim y$. By definition there exists a
$\mathcal{Q}$-chain connecting $x$ and $y$. Since the union of two connected
sets with nontrivial intersection is connected again, we may thus conclude
that $x$ and $y$ are contained in a common connected subset of
$\mathcal{O}$. Hence the equivalence class of $x$ modulo $\sim $ is contained
in the connected component of $x$.

To see the converse direction, observe that for each
$x \in \mathcal{O}$, there exists $i \in I$ with $x \in Q_{i}$, and thus
$d(x,y) \le 1 < \infty $ for all $y \in Q_{i}$, which is an open set. This
shows that all $\sim $-equivalence classes are open, and consequently closed,
since they partition $\mathcal{O}$. Hence the connected component of
$x$ is contained in the $\sim $-equivalence class.
\end{proof}

\begin{remark}
\label{rem:metric_inf_ok}
As we remarked above, we allow our metrics $d_{\mathcal{Q}}$ to assume
the value $\infty $, which constitutes a mild abuse of notation. For the
purpose of the subsequent arguments, this abuse can be justified as follows:
We will only be interested in the comparison of different coverings
$\mathcal{Q},\mathcal{P}$ of the same set $\mathcal{O}$, which amounts
to studying the coarse properties of the identity map
${\mathrm{id}}_{\mathcal{O}} : (\mathcal{O},d_{\mathcal{P}}) \to (
\mathcal{O},d_{\mathcal{Q}})$. Since every connected component of
$\mathcal{O}$ is mapped onto itself, one has
$d_{\mathcal{P}}(x,y) < \infty $ iff
$d_{\mathcal{Q}}(x,y) < \infty $. Furthermore, we will often restrict attention
to the case that $\mathcal{O}$ has finitely many connected components.

Under these restrictions, one can extend the various properties of maps
between metric spaces (understood in the strict sense) to the setting of
generalized metrics (allowed to take the value $\infty $), by requiring
that for each connected component
$\mathcal{O}_{0} \subset \mathcal{O}$, the restriction
$id_{\mathcal{O}_{0}} : (\mathcal{O}_{0},d_{\mathcal{P}}) \to (
\mathcal{O}_{0},d_{\mathcal{Q}})$ has the required properties. Then all
results from the previous section become available for the slightly more
general setting. Note in particular that
${\mathrm{id}} _{\mathcal{O}} : (\mathcal{O},d_{\mathcal{P}}) \to (
\mathcal{O},d_{\mathcal{Q}})$ is a quasi-isometry if and only if all restrictions
to connected components are, provided the number of connected components
is finite. Here we simply take minima resp. maxima of the finitely many
constants involved in the formulation of quasi-isometry statements for
each component, to obtain globally valid constants.
\end{remark}

\begin{remark}
\label{rem:ls_geodesic}
We note that each covering-induced metric $d_{\mathcal{P}}$ is $c$-large
scale geodesic, for $c=1$.
\end{remark}

The property of weak equivalence of two admissible coverings has a characterization
in terms of the coarse structure of their associated metrics. A result
similar to $i) \Rightarrow ii)$ in the following Theorem can be found in
\cite[Proposition 3.8 C)]{FeichtingerGroebnerBanachSpacesOfDistributions}.

\begin{theorem}%
\label{thm:WeakEquivQuasiIso}
Let $\mathcal{O}\subset \mathbb{R}^{d}$ be open and let
$\mathcal{Q}= (Q_{i})_{i\in I}$ and $\mathcal{P}=(P_{j})_{j\in J}$ be admissible
coverings of $\mathcal{O}$ comprised of open connected subsets of
$\mathcal{O}$. Then the following statements are equivalent:
\begin{enumerate}[ii)]
\item[i)] The coverings $\mathcal{Q}$ and $\mathcal{P}$ are weakly equivalent.
\item[ii)] The map
$\mathrm{id}:(\mathcal{O}, d_{\mathcal{Q}}) \to (\mathcal{O}, d_{
\mathcal{P}}),\ x\mapsto x$ is a quasi-isometry.
\end{enumerate}
If $\mathcal{O}$ has only finitely many connected components, $(i)$ and $(ii)$ are
equivalent to
\begin{enumerate}[(iii)]
\item[(iii)] The map
$\mathrm{id}:(\mathcal{O}, d_{\mathcal{Q}}) \to (\mathcal{O}, d_{
\mathcal{P}}),\ x\mapsto x$ is a coarse equivalence.
\end{enumerate}
\end{theorem}

\begin{proof}
$i) \Rightarrow ii)$: Assume that the coverings are weakly equivalent.
Our goal is to show the existence of $L>0$ such that
\begin{align*}
L^{{-1}}d_{\mathcal{Q}}(x,y)\leq d_{\mathcal{P}}(x,y) \leq Ld_{
\mathcal{Q}}(x,y)
\end{align*}
for all $x,y\in \mathcal{O}$. It suffices to show the right inequality,
the left inequality follows by interchanging the roles of
$\mathcal{P}$ and $\mathcal{Q}$. Let $x,y \in \mathcal{O}$. We first consider
the case that $d_{\mathcal{Q}}(x,y)=1$, this implies the existence of a
set $Q\in \mathcal{Q}$ with $x,y\in Q$. The weak subordinateness of
$\mathcal{Q}$ to $\mathcal{P}$ implies that there are
$P_{1}, \ldots , P_{K} \in \mathcal{P}$ for some
$K\leq N(\mathcal{Q}, \mathcal{P})$ (see \reftext{Definition~\ref{def:NeighborsAndEquivalence}}) with the property
$Q\cap P_{j} \neq \emptyset $ for $j\in \{1, \ldots , K\}$,
$Q\subset \bigcup _{j=1}^{K} P_{j}$ and $Q\cap P =\emptyset $ for all
$P\in \mathcal{P}\setminus \{P_{1}, \ldots , P_{K}\}$.

Define the set
\begin{align*}
C_{x}:=
\Set{ z\in Q|
\begin{array}{l}
x,z \text{ are connected by a }\mathcal{P}\text{-chain }\\
(P_{\ell}^{*})_{\ell =1}^{m} \text{ of length } m\leq K \mbox{ with } P_{\ell}^{*} \cap Q \neq \emptyset \\
\text{and } P_{\ell}^{*} \neq P_{\ell '}^{*} \text{ for } \ell \neq \ell '
\end{array}
}.
\end{align*}
As in the proof of the previous lemma, the set $C_{x}$ is relatively open
in $Q$. Next, we show that it is also relatively closed in $Q$. Let
$z_{1}$ be an element of $\overline{C_{x}}^{Q}$. Because
$\mathcal{P}$ is a covering of $\mathcal{O}$, there is
$P \in \mathcal{P}$ with $z_{1}\in P$. The set $Q\cap P$ is a relatively
open neighborhood of $z_{1}$ in $Q$, this implies the existence of some
$z\in (Q\cap P) \cap C_{x}$. Hence, there is a $\mathcal{P}$-chain
$P'_{1}, \ldots , P'_{m}$ of length $m\leq K$ that connects $x$ and
$z$.

If $P=P'_{j}$ for some $j\in \{1,\ldots , m\}$, then
$P'_{1}, \ldots , P'_{j}$ is a $\mathcal{P}$-chain of length
$j\leq K$, with $P'_{j} \cap Q \neq \emptyset $, that connects $x$, and
$z_{1}\in C_{x}$.

If $P\neq P'_{j}$ for all $j\in \{1,\ldots , m\}$, the inclusion
$\{ P_{1}',\ldots ,P_{m}' \} \subset \{ P_{1},\ldots , P_{K} \}$ together
with $P_{\ell}' \neq P_{\ell '}'$ for $\ell \neq \ell '$ implies via
the {pigeonhole} principle that $m<K$. But then
$P'_{1}, \ldots , P'_{m}, P$ is a $\mathcal{P}$-chain of length
$m+1 \leq K$ that connects $x$ and $z_{1}$. Here,
$z\in P_{m}' \cap P \neq \emptyset $. In every case,
$z_{1} \in C_{x}$ and $C_{x}$ is relatively closed in $Q$. Since $Q$ is
connected, we conclude $C_{x} = Q$ and
$d_{\mathcal{P}}(x,y)\leq K \leq N(\mathcal{Q}, \mathcal{P})$.

Next, we consider the general case
$d_{\mathcal{Q}}(x,y) = m \in \mathbb{N}$. In this case, there exists a
$\mathcal{Q}$-chain $Q_{1}, \ldots , Q_{m}$ of length $m$ that connects
$x$ and $y$. The set $Q:=\bigcup _{i=1}^{m} Q_{i}$ is connected, open and
has nonempty intersection with at most
$m N(\mathcal{Q}, \mathcal{P})$ elements of the covering
$\mathcal{P}$ (every $Q_{i}$ has nonempty intersection with at most
$N(\mathcal{Q},\mathcal{P})$ elements of the covering $\mathcal{P}$ for
$i\in \{1, \ldots , m\}$). Hence, an adaptation of the argument for the
case $m=1$ leads to the inequality
\begin{align*}
d_{\mathcal{P}}(x,y) \leq mN(\mathcal{Q}, \mathcal{P}) \leq N(
\mathcal{Q}, \mathcal{P}) d_{\mathcal{Q}}(x,y) .
\end{align*}
Interchanging of the roles of the coverings leads to
\begin{align*}
N(\mathcal{P}, \mathcal{Q})^{{-1}}d_{\mathcal{Q}}(x,y) \leq d_{
\mathcal{P}}(x,y) \leq N(\mathcal{Q}, \mathcal{P}) d_{\mathcal{Q}}(x,y).
\end{align*}
This shows that ${\mathrm{id}}$ is a quasi-isometric embedding. Since it is
also bijective, it is therefore a quasi-isometry.

$ii) \Rightarrow i)$: Assume that $\mathcal{Q}$ and $\mathcal{P}$ are not
weakly equivalent. Without loss of generality, assume that
$\mathcal{Q}$ is not weakly subordinate to $\mathcal{P}$. This implies
the existence of a sequence $(Q_{i_{n}})_{n\in \mathbb{N}}$ in
$\mathcal{Q}$ such that
\begin{align*}
\mathopen{}\mathclose{\left\lvert A_{n}\right\rvert } > N_{
\mathcal{P}}^{n+1}
\end{align*}
and
$\mathopen{\lvert }A_{n}\mathclose{\rvert } < \mathopen{\lvert }A_{n+1}
\mathclose{\rvert }$ for all $n\in \mathbb{N}$, where
\begin{equation*}
A_{n}:=\Set{j\in J| Q_{i_{n}}\cap P_{j} \neq \emptyset }%
\end{equation*}
and $N_{\mathcal{P}}$ is the admissibility constant for the covering
$\mathcal{P}$.

The next step is to show that for every $x\in Q_{i_{n}}$ and all
$j_{x}\in J$ with $x\in P_{j_{x}}$, the set
$Q_{i_{n}}\setminus P_{j_{x}}^{n*}$ is nonempty and that
$d_{\mathcal{P}}(x,y)\geq n$ for $x\in Q_{i_{n}}$ and
$y\in Q_{i_{n}}\setminus P_{j_{x}}^{n*}$ (cf. \reftext{Definition~\ref{def:SetOfNeighbors}}).

Let $x\in Q_{i_{n}}$ be arbitrary for some $n\in \mathbb{N}$. Since
$\mathcal{P}$ is a covering of $\mathcal{O}$, there is at least one
$j_{x} \in J$ such that $x\in P_{j_{x}}$. If
$Q_{i_{n}}\setminus P_{j_{x}}^{n*} = \emptyset $, then
$Q_{i_{n}} \subset P_{j_{x}}^{n*}$ and for all $j\in J$ with
$P_{j} \cap Q_{i_{n}}\neq \emptyset $, we have
$P_{j}\cap P_{j_{x}}^{n*} \neq \emptyset $, which implies
$j\in j_{x}^{(n+1)*}$. Hence,
\begin{align*}
\mathopen{\lvert }A_{n}\mathclose{\rvert } \leq \mathopen{\lvert }j_{x}^{(n+1)*}
\mathclose{\rvert }\leq N_{\mathcal{P}}^{n+1},
\end{align*}
which is a contradiction to the properties of the set $A_{n}$. It follows
$Q_{i_{n}}\setminus P_{j_{x}}^{n*}\neq \emptyset $.

Let $x\in Q_{i_{n}}$ and $y\in Q_{i_{n}}\setminus P_{j_{x}}^{n*}$. If
$d_{\mathcal{P}}(x,y)<n$, then there exists a $\mathcal{P}$-chain
$P_{1}, \ldots , P_{m}$ that connects $x$ and $y$ with $m\leq n-1$. The
definition of a $\mathcal{P}$-chain implies
$y\in P_{m} \subset P_{1}^{m*}$ and $x\in P_{j_{x}} \cap P_{1}$ implies
$P_{1} \subset P_{j_{x}}^{*}$. In conclusion, we have
$y\in P_{j_{x}}^{(m+1)*} \subset P_{j_{x}}^{n*}${,} which is a contradiction
to the choice of $y$. Hence, $d_{\mathcal{P}}(x,y) \geq n$ for all
$x\in Q_{i_{n}}$ and all $y\in Q_{i_{n}}\setminus P_{j_{x}}^{n*}$ if
$x\in P_{j_{x}}$.

Now, choose $x_{n}\in Q_{i_{n}}$ and
$y_{n} \in Q_{i_{n}}\setminus P_{j_{x_{n}}}^{n*} \neq \emptyset $, where
$x_{n} \in P_{j_{x_{n}}}$. The prior part shows
$d_{\mathcal{P}}(x_{n},y_{n}) \geq n \xrightarrow{n\to \infty}
\infty $, but $d_{\mathcal{Q}}(x_{n},y_{n})\leq 1$ since
$x_{n},y_{n} \in Q_{i_{n}} \in \mathcal{Q}$. All in all, there cannot exist
constants $L, C >0$ such that
\begin{align*}
d_{\mathcal{P}}(x,y) \leq Ld_{\mathcal{Q}}(x,y) + C
\end{align*}
for all $x,y \in \mathcal{O}$. It follows that
$\mathrm{id}:(\mathcal{O}, d_{\mathcal{Q}}) \to (\mathcal{O}, d_{
\mathcal{P}})$ is not a quasi-isometry.

The equivalence $ii) \Leftrightarrow iii)$ follows from \reftext{Proposition~\ref{prop:ce_eq_qi}} and \reftext{Remarks~\ref{rem:metric_inf_ok},
\ref{rem:ls_geodesic}}.
\end{proof}

\begin{remark}
\label{rem:covering_connected_II}
Given the role of the connectedness assumption on the elements of the coverings
in the various results leading up to the theorem, it seems reasonable
to expect that they cannot be dropped from the theorem. In fact, it is
not hard to observe that, without some structural assumptions on the coverings,
\reftext{Theorem~\ref{thm:WeakEquivQuasiIso}} cannot hold. To see this, one can pick
{an admissible} covering $\mathcal{P}$, and construct a second admissible
covering $\mathcal{Q}$ from $\mathcal{P}$ by replacing suitable pairs of
sets from $\mathcal{P}$ by their unions. If the distances of the pairs
of sets making up $\mathcal{Q}$, when measured in the $\mathcal{P}$-chain
distance $d_{\mathcal{P}}$, are unbounded, $d_{\mathcal{Q}}$ will not be
quasi-equivalent to $d_{\mathcal{P}}$. On the other hand, one easily sees
that the coverings must be weakly equivalent.

However, it is currently open whether such a counterexample can be constructed
with both $\mathcal{P}$ and $\mathcal{Q}$ fulfilling the other standing
assumptions on decomposition coverings, namely tightness and structuredness,
stipulated at the beginning of Subsection \ref{subsect:dspace}. Hence it
is not clear whether the connectedness assumption is the main ingredient
making the proof of \reftext{Theorem~\ref{thm:WeakEquivQuasiIso}} work.
\end{remark}

\section{A metric characterization of coorbit equivalence}
\label{sect:metric_coorbit}

\subsection{Word metrics on locally compact groups}
\label{subsect:word_metrics}

Throughout this section, $H$ denotes an admissible dilation group, and
$H_{0}< H$ denotes the connected component of the identity element in
$H$. Since $H$ is a Lie group, $H_{0}$ is an open subgroup. We define a
metric (which is allowed to take the value $\infty $) by picking a word
metric on $H_{0}$, and suitably extending it to $H$.

\begin{definition}
\label{defn4.1}
Let $H$ be a locally compact group and let $W\subset H$ be a unit neighborhood.
Define the map
$d_{W} : H\times H \to \mathbb{N}_{0}\cup \{\infty \}$ in the following
way
\begin{align*}
d_{W}(x,y)=
\begin{cases}
\inf \Set{m \in \mathbb{N}| x^{{-1}}y \in W^{m} } & x\neq y
\\
0& x=y,
\end{cases}
\end{align*}
where we again set $\inf \emptyset = \infty $.
\end{definition}

The expression word metric is motivated by the fact that
$d_{W}(x,y) = m$ entails that the minimal number of elements (\textit{letters})
$w_{1},\ldots , w_{m} \in W$ that make the equation
$y=x w_{1} \cdots w_{m}$ true is precisely $m$. The following fact is well-known,
and we omit the proof.

\begin{lemma}
\label{lem4.2}
Let $H$ be a locally compact group and let $W\subset H$ be a symmetric
($W=W^{{-1}}$) unit neighborhood. Then $d_{W}$ is a metric on $H$, that
is \textit{left invariant}, i.e., $d_{W}(x,y) = d_{W}(zx, zy)$ for all
$x,y,z \in H$.
\end{lemma}

As for the metric in the last section, we can impose a connectedness restriction
on $W$ in such a way that precisely the elements in the same connected
component have finite distance with respect to $d_{W}$. The proof is essentially
identical to the previous one, and therefore omitted.

\begin{lemma}%
\label{lem:GroupMetricFiniteConnectedComponent}
Let $H$ be a locally compact group and let $W\subset H$ be a symmetric
unit neighborhood with $W\subset H_{0}$, where $H_{0}$ denotes the connected
component of the neutral element in $H$. Then $d_{W}(x,y)< \infty $ if
and only if $x,y$ lie in the same connected component of $H$.
\end{lemma}

We next recall that the metric spaces $(H, d_{W})$ and $(H, d_{V})$ are
coarsely equivalent for relatively compact $W, V \subset H_{0}$. Here we
refer to \cite[Corollary 4.A.6]{Cornulier_2016}

\begin{lemma}%
\label{lem:DifferentNeighborhoodsCoarselyEquivalent}
Let $H$ be a locally compact group and let $V, W\subset H$ be relatively
compact, symmetric unit neighborhoods with $V, W\subset H_{0}$. Then
\begin{align*}
\mathrm{id}_{W}^{V}: (H, d_{W}) \to (H, d_{V}), h\mapsto h
\end{align*}
is a quasi-isometry and the spaces $(H, d_{W})$ and $(H, d_{V})$ are coarsely
equivalent.
\end{lemma}

In the terminology of \cite{Cornulier_2016}, $U$-well-spread families are
metric lattices with respect to a properly chosen word metric. The following
lemma is a further example of a folklore result for which we could not
locate a handy reference.

\begin{lemma}
\label{lem:well_spread_coarse_equiv}
Let $H$ be a locally compact group, let $W\subset H$ be a relatively compact,
symmetric unit neighborhood with $W\subset H_{0}$. Furthermore, let
$X \subset H$ be a $U$-well-spread family for some relatively compact unit
neighborhood $U \subset H_{0}$. Then $(H, d_{W})$ and
$(X, d_{W}|_{X\times X})$ are coarsely equivalent.
\end{lemma}

\begin{proof}
The identity
\begin{align*}
d_{W}|_{X\times X}(x,y) = d_{W}(x,y)
\end{align*}
for all $x,y \in X$ shows that $g:X\to H, x\mapsto x$ is a quasi-isometric
embedding. It is coarsely surjective because for every $y \in H$ exists
an $x_{y}\in X$ such that $y\in x_{y}U$, which implies
$x_{y}^{{-1}}y \in U$.

Since $W$ is symmetric, $\bigcup _{n \in \mathbb{N}} W^{n}$ is easily seen
to be an open subgroup of $H_{0}$, and connectedness of $H_{0}$ then entails
$H_{0} = \bigcup _{n \in \mathbb{N}} W^{n}$. In fact, the right hand side
is an open covering by increasing sets. Since
$\overline{U}\subset H_{0}$ is compact, this entails the existence of
$n\in \mathbb{N}$ with $U \subset W^{n}$. Hence,
\begin{align*}
d_{W}(x_{y},y) \leq n.
\end{align*}
It follows that $g$ is a quasi-isometry and that
\begin{equation*}
(H, d_{W}) \text{ and } (X, d_{W}|_{X\times X})%
\end{equation*}
are coarsely equivalent, according to \reftext{Theorem~\ref{thm:QuasiIsoCoarseEqiuv}}.
\end{proof}

The last result of this section will be useful, but has a straightforward
proof, which is therefore omitted.

\begin{lemma}%
\label{lem:topologicalGroupIsoCoarseEq}
Let $H, H'$ be locally compact groups, let $W\subset H$ be a relatively
compact, symmetric unit neighborhood with $W\subset H_{0}$ and let
$V\subset H'$ be a relatively compact, symmetric unit neighborhood with
$V\subset H_{0}'$. Assume that $\phi :H\to H'$ is a topological group isomorphism,
i.e. it is a continuous isomorphism with continuous inverse. Then
$\phi $ is a quasi-isometry from $(H, d_{W})$ to $(H', d_{V})$. In particular,
$(H, d_{W})$ and $(H', d_{V})$ are coarsely equivalent.
\end{lemma}

\subsection{Coarse equivalence of the dual orbit and the dilation group}
\label{subsect:CoarseEquivalenceOfOrbitAndGroup}

In this section, we investigate the relation between the coarse structure
induced on the dual orbit $\mathcal{O}\subset \mathbb{R}^{d}$ of an admissible
group $H\subset \mathrm{GL}(\mathbb{R}^{d})$ by means of an induced covering
on the one hand, and the coarse structure induced by a relatively compact
unit neighborhood on the group $H$ on the other. It is the main purpose
of this subsection to prove that the orbit map
\begin{equation*}
p_{\xi}:H \to \mathcal{O}, h\mapsto h^{{-T}}\xi
\end{equation*}
is a quasi-isometry, when $H$ is endowed with a suitable metric, and
$\mathcal{O}$ is endowed with the covering-induced metric.

The next result collects various topological properties of the orbit map.
%
\begin{lemma}
\label{lem:OrbitMapProperties}
For each $\xi \in \mathcal{O}$, the orbit map is continuous, surjective,
closed, open, and proper, i.{e.} $p_{\xi}^{{-1}}(K)\subset H$ is compact for
compact $K\subset \mathcal{O}$.

Furthermore, the orbit $\mathcal{O}$ has finitely many components.
\end{lemma}

\begin{proof}
We refer to \cite{FuehrContinuousWaveletTransformsWithAbelian} and
\cite[Corollary 4.1.3]{Voigtlaender2015PHD} for all properties except closedness
of $p_{\xi}$. But this property follows from properness: Let
$A \subset H$ be closed, and let
$(h_{n})_{n \in \mathbb{N}} \subset A$ with
$p_{\xi}(h_{n}) \to \eta \in \mathcal{O}$. Then
$K = \{ p_{\xi}(h_{n}) : n \in \mathbb{N} \} \cup \{ \eta \} \subset
\mathcal{O}$ is compact, and then the same follows (by properness) for
$p_{\xi}^{-1}(K) \subset H$. This set contains the sequence
$(h_{n})_{n \in \mathbb{N}}$. Hence, possibly after passing to a subsequence,
$h_{n} \to h \in H$, and then $h \in A$ by closedness of $A$. But this
finally entails
$\eta = \lim _{n \to \infty} p_{\xi}(h_{n}) = p_{\xi}(h) \in p_{\xi}(A)$,
as desired.
\end{proof}

Recall from Section~\ref{sect:metric_reformulation}, that we are primarily
interested in induced coverings consisting of connected sets. In order
to properly relate the metric arising from the induced covering with a
word metric on $H$, the condition $H_{\xi }\subset H_{0}$ will be important,
where $H_{0} \subset H$ denotes the connected component of the identity
element in $H$. The following lemma makes some basic observations following
from this assumption.

\begin{lemma}
\label{lem:conn_comps}
Assume that the isotropy group
$H_{\xi}=\Set{h\in H| p_{\xi}(h) = \xi}$ is a subset of the identity component
$H_{0} \subset H$.
\begin{enumerate}[(ii)]
\item[(i)] The orbit map $p_{\xi}$ induces a bijection between the sets
of connected components of $H$ and those of $\mathcal{O}$.
\item[(ii)] For every connected set $B\subset \mathcal{O}$, there is a
connected set $A\subset H$ such that
\begin{align*}
p_{\xi}^{{-1}}(B) \subset A.
\end{align*}
\end{enumerate}
\end{lemma}

\begin{proof}
We start the proof of (i) by noting that $H_{0}$ is a normal subgroup of
$H$. Then the factor group $H/H_{0}$ acts on the set of connected components
of $\mathcal{O}$ in a canonical way. To see this, note that if
$C \subset \mathcal{O}$ is a connected component, then $H_{0}^{T} C$ is
a connected open set with open complement (the union of the remaining
$H_{0}^{T}$-orbits), which intersects $C$ nontrivially. This implies
$H_{0}^{T} C \subset C$ by connectedness of the left-hand side, and
$H_{0}^{T} C = C$ by the remaining properties of $H_{0}^{T} C$. As a consequence,
if $\mathcal{C}$ denotes the set of connected components of
$\mathcal{O}$, the map
$(H/H_{0}) \times \mathcal{C} \ni (h H_{0}, C) \mapsto h^{-T} C \in
\mathcal{C}$ is a well-defined action of the quotient group on
$\mathcal{C}$. Furthermore, it is clearly transitive, since the action
of $H$ on $\mathcal{O}$ is transitive.

Finally, the condition $H_{\xi }\subset H_{0}$ implies that the action
of $H/H_{0}$ is free. To see this, it suffices to show that the stabilizer
of $C = H_{0}^{T} \xi $ is trivial, since the action is transitive. For
any $h \in H$ satisfying $h^{-T} H_{0}^{T} C = C$ we get
$\xi \in h H_{0}^{T} \xi $, and thus $h^{-T} h_{0}^{-T} \xi = \xi $, with
suitably chosen $h_{0} \in H_{0}$. Hence
$h h_{0} \in H_{\xi }\subset H_{0}$, which implies $h \in H_{0}$. This
finishes the proof of (i).

For the proof of (ii), we first note that by assumption on $B$, $B$ is
contained in a connected component of $\mathcal{O}$. By part (i), the inverse
image of this connected component is a connected component
$A \subset H$.
\end{proof}

Without the condition $H_{\xi }\subset H_{0}$, the orbit map is not a quasi-isometry
or even a coarse map between the group $H$ and the associated dual orbit.

\begin{corollary}
\label{cor4.9}
Let $W \subset H$ be a relatively compact, symmetric unit neighborhood
with $W\subset H_{0}$. Let $\mathcal{Q}$ be an induced connected covering
of $\mathcal{O}$ by $H$. If $H_{\xi }\not \subset H_{0}$, then
$p_{\xi}$ is not a metrically proper map from $(H, d_{W})$ to
$(\mathcal{O}, d_{\mathcal{Q}})$.
\end{corollary}

\begin{proof}
If $H_{\xi }\not \subset H_{0}$, there exist
$h_{1}, h_{2} \in H_{\xi}$ that are members of different connected components
in $H$. This implies
\begin{align*}
d_{\mathcal{Q}}(p_{\xi}(h_{1}), p_{\xi}(h_{2})) = 0 \quad \text{and}
\quad d_{W}(h_{1}, h_{2})=\infty
\end{align*}
because of $p_{\xi}(h_{1}) = p_{\xi}(h_{2})$ and \reftext{Lemma~\ref{lem:GroupMetricFiniteConnectedComponent}}. This means the inverse image
of the bounded set $A:=\{p_{\xi}(h_{1})\}$ in
$(\mathcal{O}, d_{\mathcal{Q}})$ under $p_{\xi}$ is not bounded in
$(H, d_{W})$ since $p^{{-1}}_{\xi}(A)\supset \{h_{1},h_{2}\}$.
\end{proof}

In the remaining part of this section, we show that the condition
$H_{\xi }\subset H_{0}$ ensures that the orbit map $p_{\xi}$ is indeed
a quasi-isometry for suitable $\xi \in \mathcal{O}$. As preparation for
this, we cite the next result.

\begin{lemma}[\cite{Voigtlaender2015PHD} Lemma 4.1.4.]%
\label{lem:FelixCompactSetLemma}
Let $K_{1}, K_{2} \subset \mathcal{O}$ be compact sets and define
\begin{align*}
L:=L(K_{1},K_{2}):= p_{\xi}^{{-1}}(K_{1}) \cdot H_{\xi} \cdot \left (p_{
\xi}^{{-1}}(K_{2})\right )^{{-1}}.
\end{align*}
Then $L\subset H$ is compact and for arbitrary $g,h\in H$ satisfying
\begin{equation*}
h^{{-T}}K_{1} \cap g^{{-T}}K_{2} \neq \emptyset ,%
\end{equation*}
we have $g\in hL$.
\end{lemma}

We will need the following property of the set $L$ from the above lemma.

\begin{lemma}%
\label{lem:FelixCompactSetLemmaConnected}
If $H_{\xi }\subset H_{0}$ and $K_{1}, K_{2} \subset \mathcal{O}$ are compact
sets contained in the same connected component of $\mathcal{O}$, then
$L(K_{1}, K_{2}) \subset H_{0}$ with the set $L(K_{1}, K_{2})$ from \reftext{Lemma~\ref{lem:FelixCompactSetLemma}}.
\end{lemma}

\begin{proof}
If one of the sets $K_{1}, K_{2}$ is empty, the claim is obvious. Consider
the case of nonempty $K_{1}, K_{2}$ in the following. There is a connected
component $B\subset \mathcal{O}$ with $K_{1}, K_{2} \subset B$. By \reftext{Lemma~\ref{lem:conn_comps}}, there is a connected component $A\subset H$ such
that
\begin{align*}
p_{\xi}^{{-1}}(K_{1}) , p_{\xi}^{{-1}}(K_{2}) \subset p_{\xi}^{{-1}}(B)
\subset A.
\end{align*}
This implies
\begin{align*}
L(K_{1}, K_{2}) \subset A\cdot H_{\xi }\cdot A^{{-1}}\subset A \cdot H_{0}
\cdot A^{{-1}}=:M.
\end{align*}
Since $K_{1}, K_{2}$ are nonempty, the set $A$ is nonempty as well. Hence,
there is $a\in A$, which further implies
$e=a\cdot e \cdot a^{{-1}}\in M$. As $M$ is a connected set as product
of connected sets, and contains $e$, we infer $M\subset H_{0}$.
\end{proof}

In the following statements and arguments, we will occasionally make the
assumption that $\xi \in Q$, where the set $Q$ is used in the construction
of an induced connected covering $(h_{i}^{{-T}}Q)_{i\in I}$ of
$\mathcal{O}$ by $H$. Observe that this can be achieved by relabeling the
covering, e.g. by picking $i_{0} \in I$ with
$\xi \in h_{i_{0}}^{{-T}}Q =: Q'$, and considering the system
$(g_{i}^{{-T}}Q')_{i \in I} = (h_{i}^{{-T}}Q)_{i \in I}$, where
$g_{i} = h_{i} h_{i_{0}}^{-1}$. Here it is easy to verify that
$(g_{i})_{i \in I}$ is well-spread if $(h_{i})_{i \in I}$ is.

The following inequality is the first step towards showing that the orbit map is a quasi-isometry.

\begin{lemma}%
\label{lem:OrbitMapQuasiIsoUntereUngleichung}
Assume $H_{\xi }\subset H_{0}$. Let $W \subset H$ be a relatively compact,
symmetric unit neighborhood with $W\subset H_{0}$. Let
$\mathcal{Q}= (h_{i}^{{-T}}Q)_{i\in I}$ be an induced connected covering
of $\mathcal{O}$ by $H$ with $\xi \in Q$. Then the following inequality
holds
\begin{align*}
d_{W}(g,h) \leq L d_{\mathcal{Q}}(p_{\xi}(g), p_{\xi}(h)) + C
\end{align*}
for suitable $L,C>0$ for all $g,h\in H$.
\end{lemma}

\begin{proof}
Let $g,h \in H$ be arbitrary. If
$d_{\mathcal{Q}}(p_{\xi}(g), p_{\xi}(h)) = \infty $, there is nothing to
show. If $d_{\mathcal{Q}}(p_{\xi}(g), p_{\xi}(h)) = 0$, then
\begin{equation*}
p_{\xi}(g) = p_{\xi}(h)%
\end{equation*}
and $g^{{-1}}h \in H_{\xi}$. The set $H_{\xi}\subset H_{0}$ is compact,
according to the admissibility of $H$, so that there exists a
$C_{1}\in \mathbb{N}$ with $H_{\xi}\subset W^{C_{1}}$ (cf. the proof of
\reftext{Lemma~\ref{lem:well_spread_coarse_equiv}} for a similar argument), which
implies $d_{W}(g,h) \leq C_{1}$.

If $d_{\mathcal{Q}}(p_{\xi}(g), p_{\xi}(h)) = m\in \mathbb{N}$, there exists
a $\mathcal{Q}$-chain from $p_{\xi}(g)$ to $p_{\xi}(h)$, which means there
exist $h_{i_{1}}, \ldots , h_{i_{m}}$ in $(h_{i})_{i\in I}$ such that
$p_{\xi}(g)\in h_{i_{1}}^{{-T}}Q$, $p_{\xi}(h)\in h_{i_{m}}^{{-T}}Q$ and
\begin{align*}
h_{i_{j}}^{{-T}}Q \cap h_{i_{j+1}}^{{-T}}Q \neq \emptyset
\end{align*}
for all $j\in \{1, \ldots , m-1\}$. According to \reftext{Lemma~\ref{lem:FelixCompactSetLemma}}, this implies
\begin{align*}
g^{{-1}}h_{i_{1}} &\in L(\{\xi \}, \overline{Q}),
\\
h_{i_{j}}^{{-1}}h_{i_{j+1}} &\in L(\overline{Q}, \overline{Q}) \quad
\text{for } j\in \{1,\ldots m-1\}
\\
h_{i_{m}}^{{-1}}h &\in L(\overline{Q}, \{\xi \}),
\end{align*}
which allows us to deduce
\begin{align*}
&g^{{-1}}h
\\
&= g^{{-1}}h_{i_{1}} \left (\prod _{j=1}^{m-1} h_{i_{j}}^{{-1}}h_{i_{j+1}}
\right ) h_{i_{m}}^{{-1}}h
\\
&\in L(\{\xi \}, \overline{Q})L(\overline{Q}, \overline{Q})^{m-1}L(
\overline{Q}, \{\xi \}).
\end{align*}
Because $\xi \in \overline{Q}$ and $\overline{Q}$ is a connected subset
of $\mathcal{O}$, \reftext{Lemma~\ref{lem:FelixCompactSetLemmaConnected}} implies
\begin{align*}
S:= L(\{\xi \}, \overline{Q}) \cup L(\overline{Q}, \overline{Q})
\cup L(\overline{Q}, \{\xi \}) \subset H_{0}
\end{align*}
and $S$ is compact, according to \reftext{Lemma~\ref{lem:FelixCompactSetLemma}}.
Connectedness of $H_{0}$ and $W \subset H_{0}$ symmetric, open imply that
$H_{0} = \bigcup _{n \in \mathbb{N}} W^{n}$, hence compactness of
$S$ implies $S\subset W^{L_{1}}$ for suitable $L_{1}$, which implies
\begin{align*}
g^{{-1}}h \in S^{m+1} \subset W^{L_{1} (m+1)}.
\end{align*}
We conclude
\begin{align*}
d_{W}(g, h) \leq L_{1}(m+1) = L_{1} d_{\mathcal{Q}}(p_{\xi}(g), p_{
\xi}(h)) + L_{1}.
\end{align*}

In summary, for $L:=L_{1}$ and $C:=L_{1}+C_{1}$, we have
\begin{align*}
d_{W}(g,h) \leq L d_{\mathcal{Q}}(p_{\xi}(g), p_{\xi}(h)) + C
\end{align*}
for all $g,h\in H$.
\end{proof}

In the next result, we show the remaining inequality for a special case.

\begin{lemma}%
\label{lem:SpecialCaseOne}
Let $W \subset H$ be a relatively compact, symmetric unit neighborhood
with $W\subset H_{0}$. Let $\mathcal{Q}= (h_{i}^{{-T}}Q)_{i\in I}$ be an
induced connected covering of $\mathcal{O}$ by $H$ with $\xi \in Q$ for
some well-spread family $(h_{i})_{i\in I}$ in $H$ and open, relatively
compact $Q \subset \mathcal{O}$. Then there exists a constant $C>0$ such
that for all $g,h\in H$ with $d_{W}(g,h)\leq 1$, the inequality
\begin{align*}
d_{\mathcal{Q}}(p_{\xi}(g), p_{\xi}(h)) \leq C
\end{align*}
holds.
\end{lemma}

\begin{proof}
If $g,h\in H$ with $d_{W}(g,h)\leq 1$, then $h=gw$ for some $w\in W$ (if
$d_{W}(g,h)=0$, then $w=e$). The set $W\subset H_{0}$ is relatively compact
therefore we can choose a relatively compact, connected set
$V\subset H_{0}$ with $W\subset V$ (this is possible by choosing a relatively
compact, connected unit neighborhood $U\subset H$ and using that
$W\subset \bigcup _{m=1}^{\infty }(U^{\circ})^{m} $, then $V:=U^{m}$, for
a suitable $m$, is a valid choice). We have
\begin{equation*}
\overline{p_{\xi}(V)} \subset p_{\xi}(\overline{V}) \subset
\mathcal{O}%
\end{equation*}
because the orbit map is closed according to \reftext{Lemma~\ref{lem:OrbitMapProperties}}, which shows that
$K_{1}:=\overline{p_{\xi}(V)}$ is a compact and connected subset of
$\mathcal{O}$. If we set $K_{2}:=\overline{Q}$, then \reftext{Lemma~\ref{lem:InducedCoveringNormCondition}} implies that the set
\begin{align*}
I_{g}(K_{1}, K_{2}) :=
\Set{i\in I | g^{{-T}}K_{1} \cap h_{i}^{{-T}}K_{2} \neq \emptyset}
\end{align*}
has at most $C_{2} = C_{2}(K_{1}, K_{2}, (h_{i})_{i\in I})$ elements. Crucial
for this proof is that $C_{2}$ does not depend on $g$. For arbitrary
$x \in g^{{-T}}K_{1}=:K$, consider the set
\begin{align*}
C_{x}:=
\Set{ z\in K|
\begin{array}{l}
x,z \text{ are connected by a }\mathcal{Q}\text{-chain }\\
(Q_{\ell})_{\ell =1}^{m} \text{ of length } m\leq C_{2} \mbox{ with } Q_{\ell }\cap K \neq \emptyset \\
\text{and } Q_{\ell }\neq Q_{\ell '} \text{ for } \ell \neq \ell '
\end{array}
}.
\end{align*}
As in the proof of \reftext{Theorem~\ref{thm:WeakEquivQuasiIso}}, we see that
$C_{x}=K$ because $K$ is connected as well. Since $x\in K$ was arbitrary,
we have
\begin{align*}
d_{\mathcal{Q}}(x,y) \leq C_{2}
\end{align*}
for all $x,y\in K$. In particular, it holds
\begin{align*}
p_{\xi}(g) = g^{{-T}}\xi \in g^{{-T}}W^{{-T}}\xi \subset g^{{-T}}
\overline{p_{\xi}(V)} = g^{{-T}}K_{1} = K
\end{align*}
and
\begin{align*}
p_{\xi}(gw)
= g^{{-T}}w^{{-T}}\xi \in g^{{-T}}W^{{-T}}\xi \subset g^{{-T}}
\overline{p_{\xi}(V)} = g^{{-T}}K_{1} = K,
\end{align*}
which implies
\begin{align*}
d_{\mathcal{Q}}(p_{\xi}(g), p_{\xi}(h)) \leq C_{2}.
\end{align*}
As $C_{2}$ is independent of $g, h$ with $d_{W}(g,h)\leq 1$, this concludes
the proof.
\end{proof}

This special case implies the general case.

\begin{corollary}%
\label{cor:OrbitMapQuasiIsoUpperInequality}
Let $W \subset H$ be a relatively compact, symmetric unit neighborhood
with $W\subset H_{0}$. Let $\mathcal{Q}= (h_{i}^{{-T}}Q)_{i\in I}$ be an
induced connected covering of $\mathcal{O}$ by $H$ with $\xi \in Q$, for
some open, relatively compact $Q \subset \mathcal{O}$. There exists a constant
$C>0$ such that for all $g,h\in H$, the inequality
\begin{align*}
d_{\mathcal{Q}}(p_{\xi}(g), p_{\xi}(h)) \leq C d_{W}(g,h)
\end{align*}
holds.
\end{corollary}

\begin{proof}
Let $g,h$ be arbitrary elements in $H$. If $d_{W}(g,h)=\infty $, there
is nothing to show. If $d_{W}(g,h)=0$, then $g=h$ and also the left side
of the inequality vanishes. If $d_{W}(g,h)=m\in \mathbb{N}$, we can write
\begin{align*}
h=g\prod _{k=1}^{m} w_{k}
\end{align*}
for some $w_{k} \in W$. Define $h_{n}:= g\prod _{k=1}^{n} w_{k}$ for
$n\in \{0, \ldots , m\}$. This entails
\begin{align*}
d_{W}(h_{n} , h_{n+1}) = d_{W}(e, w_{n+1}) \leq 1
\end{align*}
for all $n\in \{0, \ldots , m-1\}$, where we used the left invariance of
the metric $d_{W}$. Repeated application of the triangle inequality and
\reftext{Lemma~\ref{lem:SpecialCaseOne}} yield
\begin{align*}
d_{\mathcal{Q}}(p_{\xi}(g), p_{\xi}(h)) &\leq \sum _{n=0}^{m-1}d_{
\mathcal{Q}}(p_{\xi}(h_{n}), p_{\xi}(h_{n+1}))
\\
& \leq C m
\\
&= C d_{W}(g,h).
\end{align*}
\end{proof}

\begin{remark}
\label{rem4.15}
Note that we did not assume $H_{\xi }\subset H_{0}$ in \reftext{Lemma~\ref{lem:SpecialCaseOne}} and \reftext{Corollary~\ref{cor:OrbitMapQuasiIsoUpperInequality}}, which means that even without
this assumption, the orbit map is uniformly bornologous.
\end{remark}

The previous results combine to give the following important observation.

\begin{theorem}%
\label{thm:OrbitMapQuasiIsometry}
Assume $H_{\xi }\subset H_{0}$. Let $W \subset H$ be a relatively compact,
symmetric unit neighborhood with $W\subset H_{0}$. Furthermore, let
$\mathcal{Q}= (h_{i}^{{-T}}Q)_{i\in I}$ be an induced connected covering
of $\mathcal{O}$ by $H$ with $\xi \in Q$, for some open, relatively compact
$Q \subset \mathcal{O}$. Then
\begin{equation*}
p_{\xi}:(H, d_{W}) \to (\mathcal{O}, d_{\mathcal{Q}}), h\mapsto h^{{-T}}
\xi
\end{equation*}
is a quasi-isometry. In particular $(H, d_{W})$ and
$(\mathcal{O}, d_{\mathcal{Q}})$ are coarsely equivalent.
\end{theorem}

\begin{proof}
\reftext{Lemma~\ref{lem:OrbitMapQuasiIsoUntereUngleichung}} and \reftext{Corollary~\ref{cor:OrbitMapQuasiIsoUpperInequality}} show that
\begin{equation*}
p_{\xi}:(H, d_{W}) \to (\mathcal{O}, d_{\mathcal{Q}})%
\end{equation*}
is a quasi-isometric embedding. Since it is surjective, it is a quasi-isometry.
The coarse equivalence follows with \reftext{Theorem~\ref{thm:QuasiIsoCoarseEqiuv}}.
\end{proof}

\subsection{Characterizing coorbit equivalence}
\label{subsect:cc_eq}

Now we can state the central theorem characterizing whether two admissible
groups with the same dual orbit $\mathcal{O}$ induce weakly equivalent
coverings of $\mathcal{O}$, without the need to explicitly determine such
coverings.

The basic idea underlying the theorem is conveyed by the following commutative
diagram, and an appeal to Theorems \ref{thm:WeakEquivQuasiIso} and \ref{thm:OrbitMapQuasiIsometry}. 
\begin{center}
\begin{tikzcd}
(\mathcal{O}, d_{\mathcal{Q}}) \arrow{r}{\mathrm{id}_{\calQ}^{\calP}} & (
\mathcal{O}, d_{\mathcal{P}})
\arrow{d}{\left (p_{\xi _{2}}^{H_{2}}\right )^{\inv}}
\\
(H_{1}, d_{W}) \arrow{r}{\phi} \arrow{u}{p_{\xi _{1}}^{H_{1}}} & (H_{2},
d_{V})
\end{tikzcd}
\end{center}

\begin{theorem}%
\label{thm:MainTheoremEquivalenceOfGroups}
Let $H_{1}, H_{2} \subset \mathrm{GL}(\mathbb{R}^{d})$ be two admissible
dilation groups with dual orbit $\mathcal{O}= H_{1}^{-T} \xi _{1}$ and
$\mathcal{O}' = H_{2}^{-T} \xi _{2}$. Assume that the isotropy groups are
contained in the respective connected components of $H_{1}$ and
$H_{2}$. Let $W \subset (H_{1})_{0},V \subset (H_{2})_{0}$ denote relatively
compact, symmetric neighborhoods, let
$\mathcal{Q} = (h_{i}^{-T} Q)_{i \in I}$ and
$\mathcal{P} = (g_{j}^{-T} P)_{j \in J}$ denote induced, connected coverings
of $\mathcal{O}$ resp. $\mathcal{O}'$, for suitable choices of well-spread
families $(h_{i})_{i \in I} \subset H_{1}$,
$(g_{j})_{j \in J} \subset H_{2}$.

Denote the orbit maps by
\begin{align*}
p_{\xi _{1}}^{H_{1}} : (H_{1}, d_{W}) \to (\mathcal{O}, d_{
\mathcal{Q}}), h \mapsto h^{{-T}}\xi _{1}
\end{align*}
and
\begin{align*}
p_{
\xi _{2}}^{H_{2}} : (H_{2}, d_{V}) \to (\mathcal{O}, d_{\mathcal{P}}),
h \mapsto h^{{-T}}\xi _{2}.
\end{align*}
Let $\left (p_{\xi _{2}}^{H_{2}}\right )^{{-1}}$ denote an arbitrary right
inverse of $p_{\xi _{2}}^{H_{2}}$ (which exists since
$p_{\xi _{2}}^{H_{2}}$ is surjective). Then the following statements are
equivalent:
\begin{enumerate}[iii)]
\item[i)] $H_{1}$ and $H_{2}$ are coorbit equivalent.
\item[ii)] $\mathcal{O} = \mathcal{O}'$, and the map
\begin{equation*}
\phi := \left (p_{\xi _{2}}^{H_{2}}\right )^{{-1}}\circ p_{\xi _{1}}^{H_{1}}
: (H_{1}, d_{W}) \to (H_{2}, d_{V})%
\end{equation*}
is a coarse equivalence.
\item[iii)] $\mathcal{O} = \mathcal{O}'$, and the map
\begin{equation*}
\phi := \left (p_{\xi _{2}}^{H_{2}}\right )^{{-1}}\circ p_{\xi _{1}}^{H_{1}}
: (H_{1}, d_{W}) \to (H_{2}, d_{V})%
\end{equation*}
is a quasi-isometry.
\end{enumerate}
\end{theorem}

\begin{proof}
First, we note that any right inverse of $p_{\xi _{2}}^{H_{2}}$ is also
a coarse right inverse in the sense of \reftext{Definition~\ref{def:CoarseRightInverse}}. Since our assumptions and \reftext{Theorem~\ref{thm:OrbitMapQuasiIsometry}} imply that
$p_{\xi _{2}}^{H_{2}}(H_{2}, d_{V}) \to (\mathcal{O}', d_{\mathcal{P}})$
is a quasi-isometry, \reftext{Lemma~\ref{lem:QuasiIsoHasCoaInv}} shows that
\begin{equation*}
\left (p_{\xi _{2}}^{H_{2}}\right )^{{-1}}:(\mathcal{O}, d_{
\mathcal{P}}) \to (H_{2}, d_{V})%
\end{equation*}
is also a quasi-isometry.

$i)\Rightarrow iii)$: By \reftext{Theorem~\ref{thm:coorbit_equiv_dual_orbits}},
$H_{1}$ and $H_{2}$ are coorbit equivalent iff
$\mathcal{O} = \mathcal{O}'$, and the dual coverings induced by
$H_{1}$ and $H_{2}$ are weakly equivalent.

We can now write
\begin{equation*}
\phi = \left (p_{\xi _{2}}^{H_{2}} \right )^{{-1}}\circ \mathrm{id}_{
\mathcal{Q}}^{\mathcal{P}}\circ p_{\xi _{1}}^{H_{1}} : (H_{1}, d_{W})
\to (H_{2}, d_{V})~,%
\end{equation*}
as a composition of quasi-isometries:
$\mathrm{id}_{\mathcal{Q}}^{\mathcal{P}}$ is a quasi-isometry
$(\mathcal{O},d_{\mathcal{Q}})\to (\mathcal{O},d_{\mathcal{P}})$ by \reftext{Theorem~\ref{thm:WeakEquivQuasiIso}}, and the other two mappings are quasi-isometries
by \reftext{Theorem~\ref{thm:OrbitMapQuasiIsometry}}. Hence $\phi $ is a quasi-isometry
as well.

$iii)\Rightarrow i)$: Our assumptions imply that
$p_{\xi _{1}}^{H_{1}}$ and $p_{\xi _{2}}^{H_{2}}$ are quasi-isometries.
Furthermore, the same argument as for $p_{\xi _{2}}^{H_{2}}$ shows that
there exists a quasi-isometric right inverse
$\left (p_{\xi _{1}}^{H_{1}}\right )^{{-1}}$ of
$p_{\xi _{1}}^{H_{1}}$. We conclude that the map
\begin{align*}
p_{\xi _{2}}^{H_{2}} \circ \phi \circ \left (p_{\xi _{1}}^{H_{1}}
\right )^{{-1}}&= p_{\xi _{2}}^{H_{2}} \circ \left (\left (p_{\xi _{2}}^{H_{2}}
\right )^{{-1}}\circ \mathrm{id}_{\mathcal{Q}}^{\mathcal{P}}\circ p_{
\xi _{1}}^{H_{1}}\right ) \circ \left (p_{\xi _{1}}^{H_{1}}\right )^{{-1}}
\\
& = \left (p_{\xi _{2}}^{H_{2}} \circ \left (p_{\xi _{2}}^{H_{2}}
\right )^{{-1}}\right ) \circ \mathrm{id}_{\mathcal{Q}}^{\mathcal{P}}
\circ \left (p_{\xi _{1}}^{H_{1}}\circ \left (p_{\xi _{1}}^{H_{1}}
\right )^{{-1}}\right )
\\
& =\mathrm{id}_{\mathcal{Q}}^{\mathcal{P}}
\end{align*}
is a quasi-isometry as composition of quasi-isometric maps. \reftext{Theorem~\ref{thm:WeakEquivQuasiIso}} implies that the coverings $\mathcal{Q}$ and
$\mathcal{P}$ are weakly equivalent. This suffices to establish the coorbit equivalence
of the groups $H_{1}$ and $H_{2}$, by \reftext{Theorem~\ref{thm:WeakEquivQuasiIso}}.

In order to prove the equivalence $ii) \Leftrightarrow iii)$, we first
observe that $H_{1},H_{2}$ both have finitely many connected components, by combining \reftext{Lemmas~\ref{lem:OrbitMapProperties} and \ref{lem:conn_comps}}.

With that fact established for $H_{1}$ and $H_{2}$,
$ii) \Leftrightarrow iii)$ follows from \reftext{Proposition~\ref{prop:ce_eq_qi}} and \reftext{Remark~\ref{rem:metric_inf_ok}}, since the word
metrics are large scale geodesic.
\end{proof}

\begin{remark}%
\label{rem:MainTheorem}
The main appeal of the theorem is that it replaces the technicalities involved
with computing and comparing coverings induced by different dilation groups
by a problem in geometric group theory. Whether this new formalism actually
simplifies the task still depends on the groups at hand; the results of
Section~\ref{sect:coorbit_equiv_shearlet} present a class of groups where
this is arguably the case.

At least the theorem opens the door to the systematic application of geometric
group theory techniques to the problem of deciding coorbit equivalence.
For example, the machinery of coarse invariants of groups can be brought
to bear to derive necessary criteria for coorbit equivalence. E.g., if
$H_{1}$ and $H_{2}$ are coorbit equivalent and unimodular, and
$H_{1}$ is amenable, then $H_{2}$ must be amenable as well, by Corollary
4.F.9 of \cite{Cornulier_2016}.

It is important to realize, though, that the theorem refers to coarse properties
of a very specific map between the groups $H_{1}$ and $H_{2}$, which is
induced by the respective dual actions. In particular, Section~\ref{sect:coorbit_equiv_shearlet} will provide a rich class of examples
of pairs of groups $H_{1}$ and $H_{2}$ that are topologically isomorphic
(and thus quasi-isometric), without being coorbit equivalent; see \reftext{Remark~\ref{rem:coarse_vs_coorbit_equiv}} below.
\end{remark}

\section{Characterizing coorbit equivalence of shearlet dilation groups}
\label{sect:coorbit_equiv_shearlet}

Shearlet groups in arbitrary dimensions, and their associated coorbit spaces,
were first introduced in
\cite{DahlkeShearletCoorbitSpacesCompactlySupported}. Later on, \emph{Toeplitz
shearlets}, a variation on the shearlet construction in dimensions
$\ge 3$, were introduced in
\cite{DahlkeHaeuTesCooSpaTheForTheToeSheTra}. The next definition provides
the details. The initial definition of Toeplitz shearlets only allowed
the inclusion of isotropic dilation, corresponding to the case
$\delta = 0$ in part ii). The availability of anisotropic scaling for this
class of groups was stated later in \cite{AlbertiEtAl2017}.
%
\begin{definition}[\cite{AlbertiEtAl2017}\label{def:StandardToeplitzGroups} Example 17. and Example 18.]%
\label{def:StandardToeplitzShearletGroups}
\begin{enumerate}[ii)]
\item[i)] For
$\lambda =(\lambda _{1}, \ldots , \lambda _{d-1})\in \mathbb{R}^{d-1}$,
we define the \textit{standard shearlet group in $d$-dimensions}
$H^{\lambda}$ as the set
\begin{align*}
\Set{\epsilon \, \mathrm{diag}\left (a, a^{\lambda _{1}},\ldots , a^{\lambda _{d-1}}\right )
\begin{pmatrix}
1 & s_{1} & \ldots & s_{d-1} \\
& 1 & 0\ldots & 0	\\
& & \ddots & 0	\\
& & & 1
\end{pmatrix}
|
\begin{array}{l}
a>0, \\
s_{i}\in \mathbb{R},\\
\epsilon \in \Set{\pm 1}
\end{array}
}.
\end{align*}
\item[ii)] For $\delta \in \mathbb{R}$, we define the
\textit{Toeplitz shearlet group in $d$-dimensions} $H^{\delta}$ as the set
\begin{align*}
\Set{\epsilon \, \mathrm{diag}\left (a, a^{1-\delta},\ldots , a^{1-(d-1)\delta}\right )\cdot T(1, s_{1}, \ldots , s_{d-1})
|
\begin{array}{l}
a>0,\\
s_{i}\in \mathbb{R},\\
\epsilon \in \Set{\pm 1}
\end{array}
},
\end{align*}
where the matrix $T(1, s_{1}, \ldots , s_{d-1})$ is defined by
\begin{align*}
T(1, s_{1}, \ldots , s_{d-1}):=
\begin{pmatrix}
1 & s_{1} & s_{2} & \ldots & \ldots & s_{d-1}
\\
& 1 & s_{1} & s_{2} & \ldots & s_{d-2}
\\
&& \ddots & \ddots & \ddots &
\\
&&& 1 & s_{1} & s_{2}
\\
&&&& 1 & s_{1}
\\
&&&&& 1
\end{pmatrix}.
\end{align*}
\end{enumerate}
\end{definition}

\begin{remark}
\label{rem5.2}
The two constructions are special cases of the class of
\emph{generalized shearlet dilation groups} introduced in \cite{FuRT} and
further studied in \cite{AlbertiEtAl2017}; see \reftext{Definition~\ref{defn:shearlet_dil}} below. Note that each class actually defines a
whole family of groups, parametrized by different choices for the diagonal.

In dimension two the distinction between the standard and Toeplitz shearlet
dilation groups is moot. In dimension three, one can show that standard
and Toeplitz shearlet are the only possible generalized shearlet dilation
groups (cf. \cite{AlbertiEtAl2017}, Remark 19). With increasing dimension,
the set of generalized shearlet dilation groups besides the standard and
Toeplitz case quickly becomes untractable.

Prior to the thesis \cite{KochDoktorarbeit}, the simple question whether
standard and Toeplitz shearlets have distinct coorbit spaces was not under
consideration. \cite{KochDoktorarbeit} established that, for any dimension
$d \ge 3$, two distinct groups from the families in \reftext{Definition~\ref{def:StandardToeplitzShearletGroups}} are not coorbit equivalent. For
dimension three, this was done by computing and comparing induced coverings,
see Chapter 3 of \cite{KochDoktorarbeit} for details. The arguments employed
to establish this result nicely illustrate the computational burden imposed
by the use of induced coverings, and motivated the introduction of coarse geometry as a tool in Chapter 5 of \cite{KochDoktorarbeit}. The comparison of standard and Toeplitz shearlet groups in dimensions $>3$ was then performed using the new toolbox. Our treatment of general shearlet dilation groups is a generalization of the arguments from \cite[Section 5.5]{KochDoktorarbeit}. 
\end{remark}

\subsection{Shearlet dilation groups: definition and basic properties}
\label{sec5.1}

We now recall basic notation regarding matrix groups and their Lie algebras.
$\mathfrak{gl}(\mathbb{R}^{d})$ denotes the set of all real
$d \times d$-matrices. We let
\begin{equation*}
\exp : \mathfrak{gl}(\mathbb{R}^{d}) \to \mathrm{GL}(\mathbb{R}^{d})%
\end{equation*}
be the exponential map defined by
\begin{equation*}
\exp (A) := \sum _{k=0}^{\infty}\frac{A^{k}}{k!}~.%
\end{equation*}
Furthermore, we denote with
$T(\mathbb{R}^{d}) \subset \mathrm{GL}(\mathbb{R}^{d})$ the group of upper
triangular $d\times d$-matrices with one on their diagonals. By definition,
the Lie algebra of a closed subgroup
$H\subset \mathrm{GL}(\mathbb{R}^{d})$ is the set $\mathfrak{h}$ of all
matrices $Y$ in $\mathfrak{gl}(\mathbb{R}^{d})$ such that
$\exp (tY)\in H$ for all $t\in \mathbb{R}$.

\begin{definition}[\cite{AlbertiEtAl2017} Definition 1.]
\label{defn:shearlet_dil}
Let $H\subset \mathrm{GL}(\mathbb{R}^{d})$ be a closed, admissible dilation
group. The group $H$ is called
\textit{generalized shearlet dilation group} if there exist two closed subgroups
\begin{equation*}
S, D \subset \mathrm{GL}(\mathbb{R}^{d})%
\end{equation*}
such that
\begin{enumerate}[iii)]
\item[i)] $S$ is a connected abelian subgroup of $T(\mathbb{R}^{d})$,
\item[ii)] $D=\Set{\exp (rY) | r\in \mathbb{R}}$ is a one-parameter group, where
$Y\in \mathfrak{gl}(\mathbb{R}^{d})$ is a diagonal matrix and
\item[iii)] every $h\in H$ has a unique representation as $h=\pm ds$ for some
$d\in D$ and $s\in S$.
\end{enumerate}
$S$ is called the \textit{shearing subgroup of $H$}\index{shearing subgroup},
$D$ is called the \textit{scaling subgroup of $H$}, and $Y$ is called the
\textit{infinitesimal generator of $D$}.
\end{definition}

The article \cite{AlbertiEtAl2017} describes a systematic algebraic construction method
for shearlet groups in arbitrary dimension $d \ge 2$ that proceeds by first constructing
a shearing group $S$ from an arbitrary nilpotent associative algebra of dimension $d-1$, and then determining conditions on the diagonal complement
$D$ from certain structure constants of the nilpotent algebra. 

We denote the canonical basis of $\mathbb{R}^{d}$ with
$e_{1},\ldots , e_{d}$ and the identity matrix in
$\mathrm{GL}(\mathbb{R}^{d})$ with $I_{d}$. The next result contains information
about the structure of shearing subgroups.

\begin{lemma}[\cite{AlbertiEtAl2017} Lemma 5. and Lemma 6.]%
\label{lem:CanonicalBasis}
Let $S$ be the shearing subgroup of a generalized shearlet dilation group
$H\subset \mathrm{GL}(\mathbb{R}^{d})$. Then the following statements hold:
\begin{enumerate}[iii)]
\item[i)] There exists a unique basis $X_{2},\ldots , X_{d}$ of the Lie algebra
$\mathfrak{s}$ of $S$ with $X_{i}^{T}e_{1}=e_{i}$ for
$2\leq i \leq d$, called the canonical basis of $\mathfrak{s}$.
\item[ii)] We have $S=\Set{I_{d} + X | X\in \mathfrak{s}}$.
\item[iii)] Let $\mathfrak{s}_{k} = {\mathrm{span}}\{X_{j} : j \ge k \}$, for
$k \in \{ 2, \ldots , d \}$. These are associative matrix algebras satisfying
$\mathfrak{s}_{k} \mathfrak{s}_{\ell }\subset \mathfrak{s}_{k+\ell -1}$,
where we write $\mathfrak{s}_{m} = \{ 0 \}$ for $m >d$.
\item[iv)] $H$ is the inner semidirect product of the normal subgroup $S$ with
the closed subgroup $D \cup -D$.
\end{enumerate}
\end{lemma}

Every generalized shearlet dilation group $H$ is admissible, and the next
result shows that all of them share the same dual orbit. This ensures that
one of the basic conditions for coorbit equivalence is already fulfilled.

\begin{lemma}[\cite{AlbertiEtAl2017} Proposition 11.]%
\label{lem:OrbitShearletGroups}
Let $S\subset T(\mathbb{R}^{d})$ be a shearing subgroup and $D$ be a compatible
scaling subgroup such that
\begin{equation*}
H= DS\cup (-DS)%
\end{equation*}
is a generalized shearlet dilation group. Then the unique open dual orbit
of $H$ is given by
$\mathcal{O}=\mathbb{R}^{*} \times \mathbb{R}^{d-1}$ and the isotropy group
of $\xi \in \mathcal{O} $ with respect to the dual action is given by
$H_{\xi}=\Set{I_{d}}$.
\end{lemma}

Note that in particular, $H_{\xi }\subset H_{0}$, where $H_{0} = DS$ denotes
the connected component of the neutral element. Furthermore, the orbit
map is bijective, thus it has a unique (right) inverse. Throughout this
section, we will use the representative $\xi _{0} = (1,0,\ldots ,0)$. As
a consequence of the bijectivity of $p_{\xi_0}$, each element of $H$ is uniquely determined
by its first line.

Note that if $H =DS \cup - DS$ is a shearlet dilation group, one can replace
$D = \exp (r Y)$ be $D' = \mathbb{R}^{+} \cdot I_{d}$ to obtain the
\textit{associated abelian shearlet dilation group}
$A = D'S \cup -D'S$. The structure of abelian dilation groups was elucidated
in \cite{FuehrContinuousWaveletTransformsWithAbelian}, and the following
lemma notes some useful properties arising from this structure. 
%
\begin{lemma}
\label{lem:abel_dil_gr}
\begin{enumerate}
\item[(a)] Let $A$ denote an abelian admissible dilation group. The Lie
algebra $\mathfrak{a}$ of $A$ coincides with the associative algebra
$\mathcal{A} = {\mathrm{span}}(A)$ generated by $A$, and
$A \subset \mathcal{A}$ is the group of invertible elements in
$\mathcal{A}$. Furthermore, the dual action of $A$ on the open dual orbit
is free.
\item[(b)] Let $A,A' \subset GL(\mathbb{R}^{d})$ be abelian admissible
dilation groups with $\mathcal{A} = {\mathrm{span}}(A)$ and
$\mathcal{A}' = {\mathrm{span}}(A')$. Let
$\alpha : \mathcal{A} \to \mathcal{A}'$ denote an algebra isomorphism.
Then there exists $C \in GL(\mathbb{R}^{d})$ such that
$\alpha (X) = C^{-1} X C$ holds, for all $X \in \mathcal{A}$.
\end{enumerate}
\end{lemma}
\begin{proof}
The first statement of part (a) is noted in
\cite[Theorem 11]{FuehrContinuousWaveletTransformsWithAbelian}, whereas
the second statement is due to
\cite[Lemma 6]{FuehrContinuousWaveletTransformsWithAbelian}.

Note that a consequence of the freeness of the operation implies
$\dim (A) = \dim (\mathcal{A}) = d$. For part (b) let
$\psi : \mathcal{A} \to \mathbb{R}^{d}$, $\psi (X) = X^{T} \xi _{0}$, and
$\psi ': \mathcal{A}' \to \mathbb{R}^{d}$ be defined analogously. Then
$\psi ,\psi '$ are vector space isomorphisms: They are linear maps, whose
images contain the (open) dual orbits, by part (a). Hence they are onto,
and therefore bijective, since the dimensions coincide.

Now let $B \in GL(\mathbb{R}^{d})$,
$B = \psi \circ \alpha ^{-1} \circ (\psi ')^{-1}$, and
$X \in \mathcal{A}$. Given any $y \in \mathbb{R}^{d}$, one has
$y = \psi '(Y) = Y^{T} \xi _{0}$ with a unique $Y \in \mathcal{A}'$, and
we can compute for arbitrary $X \in \mathcal{A}$
\begin{eqnarray*}
B^{-1} X^{T} B y & = & \left ( (\psi ' \circ \alpha \circ \psi ^{-1})
\circ X^{T} \circ (\psi \circ \alpha ^{-1} \circ (\psi ')^{-1})
\right ) (y)
\\
& = & \left ( (\psi ' \circ \alpha \circ \psi ^{-1}) \circ X^{T}
\right ) (\psi (\alpha ^{-1} (Y)))~.
\end{eqnarray*}
The construction of $\psi $ and the fact that $\mathcal{A}$ is commutative
implies $\psi (XZ) = X^{T} \psi (Z)$ for all $X, Z \in \mathcal{A}$, hence
we can continue via
\begin{eqnarray*}
B^{-1} X^{T} B y & = & ( \psi ' \circ \alpha \circ \psi ^{-1})(\psi (X
\alpha ^{-1}(Y)))
\\
& = & (\psi ' \circ \alpha ) (X \alpha ^{-1}(Y))
\\
& = & \psi '( \alpha (X) Y) = \alpha (X)^{T} Y^{T} \xi _{0} = \alpha (X)^{T}
y~,
\end{eqnarray*}
where the third equation used that $\alpha $ is an algebra homomorphism.
Hence $C = B^{-T}$ has the desired properties.
\end{proof}

\begin{remark}
\label{ref:rem_sd_conj}
The associative algebra structure of $\mathfrak{s}$ is also the key to
the systematic construction of shearing subgroups. For the following facts,
we refer to \cite{AlbertiEtAl2017}, in particular to Lemma 13 and Remark
14 therein. Every nilpotent commutative associative algebra
$\mathcal{N}$ of dimension $d-1$ gives rise to the Lie algebra
$\mathfrak{s}$ of a shearing subgroup of a generalized shearlet dilation
group. The canonical basis of $\mathfrak{s}$ is constructed by picking
a Jordan-H\"older basis in $\mathcal{N}$, and applying a suitable linear
map $\mathcal{N} \to \frak{gl}(d,\mathbb{R})$. Letting $S = \exp (\mathfrak{s})$ and
$D = \mathbb{R}^{+} \cdot I_{d}$ then results in an abelian shearlet dilation
group $H = DS \cup -DS$. Furthermore, all candidates of infinitesimal generators
for $D$ can be found by solving a system of linear equations derived from
$\mathfrak{s}$, see \cite[Lemma 9]{AlbertiEtAl2017}.

It is useful to note that two different shearing subgroups
$S_{1},S_{2}$ are related to each other by conjugation iff their Lie algebras
are isomorphic as associative algebras. Here the necessity is clear; and
the sufficiency follows by \reftext{Lemma~\ref{lem:abel_dil_gr}} (b). Thus there
is a one-to-one correspondence between conjugacy classes of shearing subgroups
(or alternatively: of abelian shearlet dilation groups) in dimension
$d$ and of isomorphism classes of nilpotent commutative associative algebras
of dimension $d-1$. This illustrates the increasing richness of this class
of dilation groups, as $d$ grows. The main theorem of this section,
which characterizes coorbit equivalence within the class of shearlet dilation
groups, exhibits the relevance of the conjugacy problem also for coorbit equivalence.
\end{remark}

\begin{remark}
\label{rem:coarse_vs_coorbit_equiv}
Given a shearlet dilation group $H = DS \cup (-DS)$ with
$D = \exp (\mathbb{R} Y)$ and canonical basis $X_{2},\ldots ,X_{d}$ of
$S$, the entries of the infinitesimal generator $Y$ can be related to certain
algebraic relation. More precisely, a necessary requirement for $H$ to
be an admissible dilation group is that the first diagonal entry of
$Y$ does not vanish, by \cite[Proposition 7]{AlbertiEtAl2017}. $Y$ can
then be normalized to $1$ as first entry, i.e.,
$Y = \mathrm{diag}(1,\lambda _{2},\ldots ,\lambda _{d})$, and the remaining
entries of $Y$ occur in the \textit{structure relations}
\begin{equation*}
\forall i,j=2,\ldots ,d~:~ [Y,X_{i}] = (1-\lambda _{i}) X_{i} ~,~ [X_{i},X_{j}]
= 0~,
\end{equation*}
characterizing the Lie algebra $\mathfrak{h}$ up to isomorphism. It turns
out that this extends to the Lie group level, yielding that $H$ is topologically
isomorphic to $H' = D'S' \cup (-D'S')$ iff $D = D'$; and this clearly entails
that $H$ and $H'$ are coarsely isometric.

Observe that this characterization imposes no restrictions on $S$ and
$S'$, respectively. By contrast, \reftext{Theorem~\ref{thm:char_ce_shearlet}} establishes
the necessary criterion for coorbit equivalence that $S$ and $S'$ are related
by conjugacy, and by \reftext{Remark~\ref{ref:rem_sd_conj}}, this entails that the
Lie algebras $\mathfrak{s},\mathfrak{s}'$ must be isomorphic as
\textit{associative} algebras. In particular, by taking $H,H'$ as standard
resp. Toeplitz shearlet dilation groups in dimension $d \ge 3$, with identical
diagonal groups $D=D'$, we obtain pairs of dilation groups that are coarsely
isometric, but not coorbit equivalent. See also \reftext{Remark~\ref{rem:comp_stand_Toep}} below.

To our knowledge, a classification of semidirect products of the type
$\mathbb{R}^{m} \rtimes \exp (\mathbb{R} Y)$ up to quasi-isometry is not
available. A subclass of these groups, for which a rich theory has been
developed, is given by the so-called Heintze groups \cite{MR353210}. Even
here, a classification up to quasi-isometry does not seem to be in sight.
These observations make clear that coarse geometric methods generally do
not provide readily applicable solutions for the problem of classifying
dilation groups up to coorbit equivalence. They also {emphasize} that coorbit
equivalence of two dilation groups depends on the coarse geometric properties
of a very specific map relating the dual actions of these groups.
\end{remark}

We can now formulate a general characterization of coorbit equivalence
for shearlet groups. The following theorem generalizes the characterization
for two-dimensional shearlet groups obtained in
\cite{Voigtlaender2015PHD}, and the results for standard and Toeplitz shearlet
groups in \cite{KochDoktorarbeit}.
%
\begin{theorem}%
\label{thm:char_ce_shearlet}
Let
$H = DS \cup - DS,H' = D'S' \cup -D'S' \subset GL(\mathbb{R}^{d})$ denote
two shearlet dilation groups. Then the following properties are equivalent:
\begin{enumerate}
\item[(a)] $H$ and $H'$ are coorbit equivalent.
\item[(b)] $D = D'$, and there exists a matrix
$C \in GL(\mathbb{R}^{d})$ with the following properties:
$C^{-1} S C = S'$, and the conjugation actions of $D$ and $C$ on $S$ commute,
i.e.
%
\begin{equation}
\label{eqn:conj_commute}
\forall s \in S \forall d \in D~:~ d^{-1} C^{-1} s C d = C^{-1} d^{-1}
s d C~.
\end{equation}
\end{enumerate}
\end{theorem}

\begin{remark}
\label{rem:comp_stand_Toep}
We can apply the theorem to the two families of shearlet dilation groups
introduced in \reftext{Definition~\ref{defn:shearlet_dil}}. We first observe that
the theorem explicitly states that different elements in the same family
(standard or Toeplitz), i.e., pairs of shearlet dilation groups that differ
only in the scaling subgroup, are not coorbit equivalent.

The comparison of elements from different families leads to the question
of conjugacy, which in turn is closely related to the associative structures
of the respective Lie algebras. Let $\mathfrak{s}_{s}$ denote the Lie algebra
of the shearing subgroup of the standard shearlet dilation group in dimension
$d \ge 3$, and let $\mathfrak{s}_{T}$ denote its Toeplitz counterpart.
Then any two elements $X,Y \in \mathfrak{s}_{s}$ fulfill $XY = 0$. By contrast,
$\mathfrak{s}_{T}$ is generated as an associative algebra by the element
$X_{2}$ of its canonical basis, showing that
$\mathfrak{s}_{T} \cong \mathbb{R}[X]/(X^{d-1})$. In particular,
$\mathfrak{s}_{s}$ and $\mathfrak{s}_{T}$ are not isomorphic as associative
algebras, and hence not conjugate.

In summary, any pair of distinct groups from \reftext{Definition~\ref{defn:shearlet_dil}} has distinct coorbit spaces. This fact was first
established in \cite{KochDoktorarbeit}.
\end{remark}

\subsection{Proof of \reftext{Theorem~\ref{thm:char_ce_shearlet}}}
\label{sec5.2}

Throughout this subsection, we will denote $H = DS \cup (-DS) $,
$H' = D' S' \cup (-D'S')$, with associated Lie algebras
$\mathfrak{h},\mathfrak{h'},\mathfrak{s},\mathfrak{s}'$, and canonical
bases $Y,X_{2},\ldots ,X_{d} \in \mathfrak{h}$ and
$Y',X_{2}',\ldots , X_{d}' \in \mathfrak{h}'$. By Proposition 7 of
\cite{AlbertiEtAl2017}, we can normalize $Y,Y'$ so that
\begin{equation*}
Y = {\mathrm{diag}}(1,\lambda _{2},\ldots ,\lambda _{d})~, ~Y' = {\mathrm{diag}}(1,
\lambda _{2}',\ldots ,\lambda _{d}')~.
\end{equation*}
We use
$\xi _{0} = (1,0,\ldots ,0) \in \mathcal{O} = \mathbb{R}^{*} \times
\mathbb{R}^{d-1}$. We parametrize group elements of $H$ and $H'$ as follows:
Given $r \in \mathbb{R}$ and
$t = (t_{2},\ldots ,t_{d}) \in \mathbb{R}^{d-1}$, we let
%
\begin{eqnarray}
\label{eqn:def_h}
h(r,t) & = & \exp (-rY) \left (I_{d}+ \sum _{j=2}^{d} t_{j} X_{j}
\right )^{-1} \in H~,~
\\
\nonumber
h'(r,t) & = & \exp (-rY') \left (I_{d}+ \sum _{j=2}^{d} t_{j} X_{j}'
\right )^{-1} \in H'~.
\end{eqnarray}
Note that $D = \{ h(r,0) : r \in \mathbb{R} \}$ and
$ S = \{ h(0,t) : t \in \mathbb{R}^{d-1} \}$, and analogous statements
hold for $D',S'$.

We can then compute
\begin{equation*}
p_{\xi _{0}}^{H} (h(r,t)) = \exp (rY) (I_{d}+ \sum _{j=2}^{d} t_{j} X_{j}^{T})
\left (
\begin{array}{c}
1
\\
0
\\
\vdots
\\
0
\end{array}
\right ) = \left (
\begin{array}{c}
e^{r}
\\
e^{\lambda _{2}r} t_{2}
\\
\vdots
\\
e^{\lambda _{d} r} t_{d}
\end{array}
\right ) ~.
\end{equation*}
The analogous calculation for $H'$ yields
\begin{equation*}
p_{\xi _{0}}^{H'} (h'(r,t)) = \left (
\begin{array}{c}
e^{r}
\\
e^{\lambda _{2}'r} t_{2}
\\
\vdots
\\
e^{\lambda _{d}' r} t_{d}
\end{array}
\right )
\end{equation*}
Throughout this section, let
$\phi = (p_{\xi _{0}}^{H'})^{-1} \circ p_{\xi _{0}}^{H} : H \to H'$. Then
a comparison of the formulae for
$p_{\xi _{0}}^{H} (h(r,t)), p_{\xi _{0}}^{H'}(h'(r,t))$ yields
%
\begin{equation}
\label{eqn:phi_shearlet_concrete}
\phi (h(r,t)) = h'(r,t')~,~ t' = (\exp (r(\lambda _{2}-\lambda _{2}'))
t_{2},\ldots , \exp (r (\lambda _{d}-\lambda _{d}')) t_{d})~.
\end{equation}
Setting $r=0$ in this formula shows that the restriction
$\phi |_{S}$ in fact maps $S$ bijectively onto $S'$, and that it is independent
of the choices of $Y,Y'$.

We can now make the first step towards the proof of the theorem, which
consists in comparing the diagonals of $H$ and $H'$.
%
\begin{lemma}
\label{lem:shear_diag_gleich}
If $H$ and $H'$ are coorbit equivalent, then $Y = Y'$, and thus
$D = D'$.
\end{lemma}
\begin{proof}
Towards a proof by contradiction, assume that $H$ and $H'$ are coorbit
equivalent, yet there exists $2 \le i \le d$ such that
$\lambda _{i} \neq \lambda _{i}'$. Let $\phi : H \to H'$ denote the map
from Equation \reftext{(\ref{eqn:phi_shearlet_concrete})}. Let
$e_{i} \in \mathbb{R}^{d-1}$ denote the vector with entry one at the
$i-1$st position, and zeros {elsewhere}. Define the sequence
$(h_{n})_{n \in \mathbb{N}}$ as $h_{n} = h(1,e_{i})^{n}$, and
$h_{n}' = \phi (h_{n})$. If $d: H \times H \to \mathbb{R}_{0}^{+}$ denotes
a word metric on $H$, we get by left-invariance that
$d(h_{n},h_{n+1}) = d(e_{H}, h)$. We need to prove that $\phi $ is not
a quasi-isometry with respect to a suitable word metric $d'$ on $H'$. For
this it is now sufficient to show that
$(d'(h_{n}',h_{n+1}'))_{n \in \mathbb{N}} \subset \mathbb{R}^{+}$ is unbounded,
or equivalently, that the sequence $(h_{n+1}')^{-1} h_{n}' \in H'$,
$n \in \mathbb{N}$, is not contained in a relatively compact set. This
will be shown by concentrating on suitably chosen matrix entries of these
products.

We start out by noting that $h(1,e_{i}) = h(1,0) h(0,e_{i})$. For the computation
of higher powers of this group element, we introduce
$T' = \{ 0 \}^{i} \times \mathbb{R}^{d-i-1}$. Then
$t = t_{i} e_{i} + t'$, with $t' \in T'$, is a vector with $i-2$ leading
zeroes, followed by $t_{i}$, followed by the remaining entries of
$t'$.

With this notation, we get
\begin{equation*}
h(-1,0) h(0,t_{i} e_{i} + t')^{-1} h(1,0) = h(0,e^{1-\lambda _{i}} t_{i}
e_{i} + t'')^{-1}~,
\end{equation*}
with suitably chosen $t'' \in T'$ depending linearly on $t'$. To see this,
it is sufficient to note that both sides of the equation are elements of
$H$ whose first lines can be made to coincide by properly choosing
$t'' \in T'$, and this first line uniquely determines the elements of
$H$ by \reftext{Lemma~\ref{lem:OrbitShearletGroups}} and the comments thereafter.
Since conjugation by $h(-1,0)$ is a group isomorphism, we also get
%
\begin{equation}
\label{eqn:conjugation_shearlet}
h(-1,0) h(0,t_{i} e_{i} + t') h(1,0) = h(0,e^{1-\lambda _{i}} t_{i} e_{i}+t'')~,
\end{equation}
and thus for all $n \in \mathbb{N}$:
\begin{equation*}
h(n,0) h(0,e_{i}) h(-n,0) = h(0,e^{-n(1 - \lambda _{i})} e_{i})~.
\end{equation*}
Furthermore, the fact that
$X_{i} X_{j} \in {\mathrm{span}} \{ X_{k} : k > \max (i,j) \}$ yields in particular,
for all $p,q \in \mathbb{R}$ and $t_{1}',t_{2}' \in T'$
\begin{equation*}
h(0,p e_{i} + t_{1}') h(0,q e_{i} + t_{2}') = h(0,(p+q)e_{i}+t_{3}')~,
\end{equation*}
with suitably chosen $t_{3}' \in T'$. In addition, utilizing the Neumann
series for the inverse of a matrix allows to compute for every
$p \in \mathbb{R}$ and every $t_{1}' \in T'$ that
\begin{equation*}
h(0,pe_{i} + t_{1}')^{-1} = h(0,-p e_{i} + t_{2}')
\end{equation*}
with suitably chosen $t_{2}' \in T'$. We observe that an analogous formula
holds for $H'$ as well.

Let us now first consider the case $\lambda _{i} \neq 1$. Then the above
formulae allow to show inductively that
%
\begin{equation}
\label{eqn:formula_hn}
h^{n} = h(n,0) h\left (0,
\frac{e^{n(1-\lambda _{i})}-1}{e^{1-\lambda _{i}}-1} e_{i} + t_{n}'
\right )~,
\end{equation}
with suitably chosen $t_{n}' \in T'$. Plugging this into \reftext{(\ref{eqn:phi_shearlet_concrete})}
yields the following fairly explicit formula
\begin{equation*}
h_{n}' = \phi (h_{n}) = h'(n,\exp (n(\lambda _{i}-\lambda _{i}'))
\cdot \frac{e^{n(1-\lambda _{i})}-1}{e^{1-\lambda _{i}}-1} e_{i} + s_{n}')
~,
\end{equation*}
with suitably chosen $s_{n}' \in T'$. In order to prove that the sequence
of matrices $(h_{n}')^{-1} h_{n+1}'$ is not bounded in $H'$, we compute
\begin{eqnarray*}
\lefteqn{ (h_{n}')^{-1} h_{n+1}'}
\\
& = & h'\left (0,\exp (n(\lambda _{i}-\lambda _{i}')) \cdot
\frac{e^{n(1-\lambda _{i})}-1}{e^{1-\lambda _{i}}-1} e_{i} + s_{n}'
\right )^{-1} h'(1,0) \cdot
\\
& & h'\left (0,\exp ((n+1)(\lambda _{i}-\lambda _{i}')) \cdot
\frac{e^{(n+1)(1-\lambda _{i})}-1}{e^{1-\lambda _{i}}-1}e_{i} + s_{n+1}'
\right )
\\
& = & h'(1,0) h(0,r_{n} e_{i} + u_{n})
\end{eqnarray*}
where $u_{n} \in T'$, and
\begin{eqnarray*}
r_{n} & = & - e^{1-\lambda _{i}'} \exp (n(\lambda _{i}-\lambda _{i}'))
\cdot \frac{e^{n(1- \lambda _{i})}-1}{e^{1-\lambda _{i}}-1} + \exp ((n+1)(
\lambda _{i}-\lambda _{i}')) \cdot
\frac{e^{(n+1) (1-\lambda _{i})}-1}{e^{1-\lambda _{i}}-1}
\\
& = &
\frac{\exp (n(\lambda _{i} - \lambda _{i}'))}{e^{1-\lambda _{i}}-1}
\left (-e^{1-\lambda _{i}'+n(1-\lambda _{i})} + e^{1-\lambda _{i}'}+ e^{
\lambda _{i} - \lambda _{i}' + (n+1)(1-\lambda _{i})} - e^{\lambda _{i}-
\lambda _{i}'} \right )
\\
& = &
\frac{\exp (n(\lambda _{i} - \lambda _{i}'))}{e^{1-\lambda _{i}}-1}
\left ( e^{1-\lambda _{i}'} - e^{\lambda _{i} - \lambda _{i}'}
\right ) ~.
\end{eqnarray*}
Now the assumptions $\lambda _{i}' < \lambda _{i}$ and
$\lambda _{i} \neq1$ allow to conclude directly that
$|r_{n}| \to \infty $, and thus the desired conclusion can be drawn. In
the case $\lambda _{i}'> \lambda _{i}$ we can conclude similarly, by considering
the sequence $h_{n} = h^{-n}$. Hence the case
$\lambda _{i} \neq 1$ is settled. However, in the case
$\lambda _{i} = 1$, the assumption
$\lambda _{i}' \neq \lambda _{i}$ yields that
$\lambda _{i}' \neq 1$, and exchanging the roles of $H$, $H'$ in the
above argument allows again to conclude that $H$ and $H'$ are not coorbit
equivalent.
\end{proof}

The next result looks at the restriction of $\phi $ to the shearing subgroup.
Throughout the following proof, we will call a map
$\alpha : U \to V$ defined between different vector spaces $U,V$
\textit{polynomial} if the coordinates of $\alpha (X)$ with respect to any
fixed basis of $V$ depend polynomially on the coordinates of $X$, taken
with respect to any fixed basis of $U$. Clearly, this notion is independent
of the choice of bases.
%
\begin{lemma}
\label{lem:shear_conj}
If $H$ and $H'$ are coorbit equivalent, there exists a matrix
$C \in GL(\mathbb{R}^{d})$ satisfying the following two properties:
\begin{enumerate}[(ii)]
\item[(i)] The map $\phi $ from Equation \reftext{(\ref{eqn:phi_shearlet_concrete})}
fulfills $\phi (s) = C^{-1} s C $, for all $s \in S$.
\item[(ii)] The conjugation actions of $D$ and $C$ commute.
\end{enumerate}
\end{lemma}
\begin{proof}
For the proof of (i), first note that equation \reftext{(\ref{eqn:phi_shearlet_concrete})}
implies that $\phi |_{S}$ maps into $S'$. We will first show that
$\phi |_{S} : S \to S'$ is in fact a group homomorphism.

Let $\mathcal{A} = {\mathrm{span}} (\{ I_{d} \} \cup \mathfrak{s} )$ and
$\mathcal{A}' = {\mathrm{span}} (\{ I_{d} \} \cup \mathfrak{s}' )$. These are
the Lie algebras of the abelian shearlet dilation groups $A,A'$ associated
to $H,H'$. Note that $p_{\xi _{0}}^{A} = \psi \circ \iota $, where
$\iota $ is the inversion map $h \mapsto h^{-1}$, and
$\psi : \mathcal{A} \to \mathbb{R}^{d}$ is the linear bijective map defined
in the proof of \reftext{Lemma~\ref{lem:abel_dil_gr}}. In the same way, we obtain
$p_{\xi _{0}}^{A'} = \psi ' \circ \iota $, with analogously defined linear
isomorphism $\psi '$.

It follows that
$\phi _{A,A'} = (p_{\xi _{0}}^{A'})^{-1} \circ p_{\xi _{0}}^{A}=
\iota \circ (\psi ')^{-1} \circ \psi \circ \iota $. The same then holds
for the restriction of this map to $S$. But by the remark immediately following
\reftext{(\ref{eqn:phi_shearlet_concrete})}, this restriction coincides with
$\phi |_{S}: S \to S'$, the restriction of the map $\phi $ associated to
$H$ and $H'$.

Now fix $s \in S$, and define the map
\begin{equation*}
\alpha : \mathfrak{s} \to H'~,~ \alpha (X) = \phi ((I_{d}+X) s)^{-1}
\phi (I_{d} +X)~.
\end{equation*}
Then $\alpha $ is a polynomial map. To see this, observe that inversion
is polynomial on the set of unipotent matrices, since the inverse of a
unipotent map can be computed by finitely many terms of a Neumann series.
Since $s$, $I_{d} +X$ and $\phi ((I_{d} +X) s)$ are all unipotent, this
finally yields that $\alpha $ is indeed a composition of polynomial maps,
hence polynomial.

Assuming that $H,H'$ are coorbit equivalent, we have that $\phi $ is a
quasi-isometry with respect to suitable word metrics $d_{H},d_{H'}$ on
$H,H'$ respectively, say with constants $a>0, b\ge 0$. Then left-invariance
of the metrics provides
\begin{eqnarray*}
d_{H'}(\alpha (X),e_{H'}) & = & d_{H'}(\phi ((I_{d}+X) s)^{-1}\phi (I_{d}+X),e_{H'})
\\
& = & d_{H'} (\phi (I_{d}+X), \phi ((I_{d}+X) s) \le a d_{H}((I_{d}+X)),(I_{d}+X)
s) + b
\\
& \le & a d_{H}(e_{H},s) + b~,
\end{eqnarray*}
which is independent of $X \in \mathfrak{s}$. Since the metric
$d_{H}$ is proper, it follows that
$\alpha : \mathfrak{s} \to \mathbb{R}^{d \times d}$ is a bounded polynomial
function, hence constant. In particular, we get
\begin{equation*}
\phi ((I_{d} + X)s)^{-1} \phi (I_{d} + X) = \alpha (X) = \alpha (0) =
\phi (s)^{-1}~,
\end{equation*}
or equivalently
\begin{equation*}
\phi ((I_{d}+X) s) = \phi (I_{d}+X) \phi (s)~.
\end{equation*}
Hence $\phi |_{S}$ is indeed a group homomorphism, and thus a group isomorphism
$S \to S'$. This implies that
$ (\psi ')^{-1} \circ \psi |_{A} : A \to A'$ is a group isomorphism. Since
$A \subset \mathcal{A}$ is dense, it follows that the linear isomorphism
$ (\psi ')^{-1} \circ \psi : \mathcal{A} \to \mathcal{A}'$ is an isomorphism
of associative algebras. Now \reftext{Lemma~\ref{lem:abel_dil_gr}}(b) applies to
provide a matrix $C \in GL(\mathbb{R}^{d})$ such that, for all
$s \in S$, $\phi (s) = C^{-1} s C$.

Proof of (ii): Part (i) has already established the existence of the matrix
$C$ conjugating $S$ into $S'$. With the notation established in equation
\reftext{(\ref{eqn:def_h})}, and using the fact that $Y=Y'$, we thus obtain for all
$r \in \mathbb{R}, t \in \mathbb{R}^{d-1}$:
\begin{equation*}
\phi (h(r,t)) = h(r,0) C^{-1} h(0,t) C~.
\end{equation*}
It remains to establish that the conjugation actions of $C$ and $D$ on
$S$ commute, if $\phi $ is a quasi-isometry. In order to see this, fix
$d = h(r,0) \in D$, and consider the map
\begin{equation*}
\beta : \mathbb{R}^{d-1} \ni t \mapsto \phi (h(0,t) d)^{-1} \phi (h(0,t))~.
\end{equation*}
Using that
\begin{equation*}
\phi (h(0,t) d) = \phi ( d \underbrace{d^{-1} h(0,t) d}_{\in S}) = d C^{-1}
d^{-1} h(0,t) d C ~,
\end{equation*}
we get
\begin{equation*}
\beta (t) = \left ( d C^{-1} d^{-1} h(0,t) d C \right )^{-1} C h(0,t) C^{-1}
= C^{-1} d^{-1} h(0,t)^{-1} d C d^{-1} C h(0,t) C^{-1}~.
\end{equation*}
Here $C$ and $d$ are fixed matrices, and the map
$t \mapsto h(0,t)^{-1}$ is polynomial, since $h(0,t)$ is unipotent. Hence
$\beta $ is a polynomial map.

In addition, if $d_{H}$ is any word metric on $H$, left-invariance yields
$d_{H}(h(0,t)d, h(0,t)) = d_{H}(d,e_{H})$, independent of $t$. Hence the
assumption that $\phi $ is a quasi-isometry implies that $\beta $ is a
bounded map, and therefore constant, i.e.
\begin{equation*}
\beta (t) = \beta (0) = d^{-1}.
\end{equation*}
Hence we obtain for all $s \in S$:
\begin{eqnarray*}
d^{-1} & = & \phi (e_{S} d)^{-1} \phi (e_{S})
\\
& = & \phi (sd)^{-1} \phi (s)
\\
& = & \phi ( d (d^{-1} s d))^{-1} \phi (s)
\\
& = & \left ( d (C^{-1} d^{-1} s d C) \right ) ^{-1} (C^{-1} s C)
\\
& = & (C^{-1} d^{-1} s^{-1} d C) d^{-1} (C^{-1} s C)
\\
& = & d^{-1} (d C^{-1} d^{-1} s^{-1} d C d^{-1}) (C^{-1} s C)~,
\end{eqnarray*}
or, equivalently,
\begin{equation*}
(C^{-1} s C) = (d C^{-1} d^{-1} s d C d^{-1}) ~,
\end{equation*}
and finally
\begin{equation*}
d^{-1} C^{-1} s C d = C^{-1} d^{-1} s d C~,
\end{equation*}
as desired.
\end{proof}

\textit{Finishing the proof of \reftext{Theorem~\ref{thm:char_ce_shearlet}}:}
$(a) \Rightarrow (b)$ is already established by the previous two lemmas.

$(b) \Rightarrow (a)$: Since $Y=Y'$, we have
$\phi (ds) = d C^{-1} s C$, for all $d \in D$. We claim that equation \reftext{(\ref{eqn:conj_commute})}
implies that $\phi $ is in fact a group homomorphism:
\begin{eqnarray*}
\phi (d_{1} s_{1} d_{2} s_{2}) & = & \phi (\underbrace{d_{1} d_{2}}_{
\in D} \underbrace{(d_{2}^{-1} s_{1} d_{2}) s_{2}}_{\in S} )
\\
& = & (d_{1} d_{2}) (C^{-1} d_{2}^{-1} s_{1} d_{2} s_{2} C)
\\
& = & (d_{1} d_{2}) (C^{-1} d_{2}^{-1} s_{1} d_{2} C) (C^{-1} s_{2} C)
\\
& \stackrel{\text{\reftext{(\ref{eqn:conj_commute})}}}{=} & (d_{1} d_{2}) (d_{2}^{-1} C^{-1}
s_{1} C d_{2}) (C^{-1} s_{2} C)
\\
& = & (d_{1} C^{-1} s_{1} C) (d_{2} C^{-1} s_{2} C)
\\
& = & \phi (d_{1} s_{1}) \phi (d_{2} s_{2})~.
\end{eqnarray*}
Hence the bijective map $\phi $ is a topological isomorphism, and therefore
a quasi-isometry. This implies that $H$ and $H'$ are coorbit equivalent.\qed

\section{Examples and applications}
\label{sect:Examples}

In this short section, we present a variety of examples selected to further illustrate
the scope and usefulness of the techniques presented in this paper. Clearly,
the full list of coverings and admissible groups listed in the introduction
of this paper can be fed into the machinery, as soon as the associated
metrics are computed.

\subsection{$\alpha $-modulation spaces and Besov spaces in dimension 1}
\label{sec6.1}

$\alpha $-modulation spaces were introduced by Gr\"obner
\cite{Groebner1992PhD}. In the following, we use results from
\cite{BorupNielsenFrameDecompositionOfDecOfSpaces} to describe the associated
coverings.

Given $0 \le \alpha < 1$, the $\alpha $-modulation covering of the real
line is given by
\begin{equation*}
\mathcal{P} = (P_{k})_{k \in \mathbb{Z}}~, P_{k} = B(k|k|^{\beta}, r |k|^{
\beta})~,
\end{equation*}
where $B(x,r)$ denotes the ball with center $x$ and radius $r$, with respect
to the standard metric, $\beta = \frac{\alpha}{1-\alpha}$, and
$r>r_{1}$, where $r_{1}$ is suitably chosen (see Section 6.1 of
\cite{BorupNielsenFrameDecompositionOfDecOfSpaces}). With this definition
in hand one sees that, up to coarse equivalence, the metric induced by
the covering can be determined for $s<t$ by
\begin{eqnarray*}
d_{\alpha}(s,t) & \asymp & 1 + \sharp \{ k \in \mathbb{Z} : s \le k |k|^{
\beta }\le t \}
\\
& \asymp & 1+ \left | |t|^{1-\alpha} - {\mathrm{sign}} (st) \cdot |s|^{1-
\alpha} \right | ~.
\end{eqnarray*}
It is straightforward to see from this that different choices of
$\alpha $ result in metrics that are not quasi-equivalent, and thus the
associated decomposition spaces, the $\alpha $-modulation spaces initially
defined in \cite{Groebner1992PhD}, are actually distinct. While we expect
this observation to be known, we have not been able to locate an easily
citable source for it.

\subsection{Wavelet coorbit spaces in dimension 2}
\label{sec6.2}

In dimension two, the admissible dilation groups have been classified up
to conjugacy and finite index subgroups in \cite{FuDiss}; the following
is a complete list of representatives, with their open dual orbits:
\begin{itemize}
\item {\textbf{Diagonal group}}
\begin{equation*}
D = \left \{ \left (
\begin{array}{c@{\quad}c}
a & 0
\\
0 & b
\end{array}
\right ) : ab \neq 0 \right \} ~,
\end{equation*}
with $\mathcal{O} = (\mathbb{R^{*}})^{2}$.
\item {\textbf{Similitude group}}
\begin{equation*}
H = \left \{ \left (
\begin{array}{c@{\quad}c}
a & b
\\
-b & a
\end{array}
\right ) : a^{2}+b^{2} > 0 \right \}
\end{equation*}
with $\mathcal{O} = \mathbb{R}^{2} \setminus \{ 0 \}$.
\item {\textbf{Shearlet group(s)}} For a fixed parameter
$c \in \mathbb{R}$,
\begin{equation*}
S_{c} = \left \{ \pm \left (
\begin{array}{c@{\quad}c}
a & b
\\
0 & a^{c}
\end{array}
\right ) : a > 0 \right \} ~,
\end{equation*}
with $\mathcal{O} = \mathbb{R}^{*} \times \mathbb{R}$.
\end{itemize}
No pair of distinct groups from this list is coorbit equivalent: The dual
orbit criterion from \reftext{Theorem~\ref{thm:MainTheoremEquivalenceOfGroups}} implies
this for the comparison of $D$ and $H$, and of either group with any of
the shearlet groups. Furthermore, \reftext{Theorem~\ref{thm:char_ce_shearlet}} applied
to the shearlet family shows that distinct members of the shearlet family
are also not coorbit equivalent. This fact was first observed in
\cite{Voigtlaender2015PHD}.

For the sake of completeness, we list the remaining classes of decomposition
spaces so far considered in dimension two:
\begin{itemize}
\item $\alpha $-modulation spaces for $0 \le \alpha < 1$, see
\cite{Groebner1992PhD};
\item curvelet and shearlet smoothness spaces
\cite{BorupNielsenFrameDecompositionOfDecOfSpaces,LabateShearletSmoothnessSpaces};
\item wave packet smoothness spaces \cite{BytVoi};
\item homogeneous and inhomogeneous anisotropic Besov spaces
\cite{Bo,FuCh}.
\end{itemize}

A full classification of these decomposition spaces is obtainable by determining coarse equivalence classes among the associated covering-induced metrics on $\mathbb{R}^2$ (or suitable open subsets $\mathcal{O} \subset \mathbb{R}^2$). 

\subsection{Shearlet coorbit spaces in higher dimensions}
\label{sec6.3}

Besides the already mentioned cases of standard and Toeplitz shearlet dilation
groups, there exists an increasing choice of alternative constructions
with distinct coorbit spaces. The remaining examples in dimension four
are described, up to conjugacy, as follows. We refer to
\cite[Example 20]{AlbertiEtAl2017} for more details. Given
$\alpha \in \{ -1,0,1 \}$, we let
\begin{equation*}
X_{2} = \left (
\begin{array}{c@{\quad}c@{\quad}c@{\quad}c}
0 & 1 & 0 & 0
\\
0 & 0 & 0 & 1
\\
0 & 0 &0 & 0
\\
0 &0 & 0 & 0
\end{array}
\right ) ~,~ X_{3,\alpha} = \left (
\begin{array}{c@{\quad}c@{\quad}c@{\quad}c}
0 & 0 & 1 & 0
\\
0&0 & 0& 0
\\
0& 0& 0& \alpha
\\
0&0 &0& 0%
\end{array}
\right )~,~ X_{4} = \left (
\begin{array}{c@{\quad}c@{\quad}c@{\quad}c}
0& 0 & 0& 1
\\
0 & 0& 0& 0
\\
0& 0& 0& 0
\\
0& 0& 0& 0%
\end{array}
\right ).
\end{equation*}
Let $\mathfrak{s}_{\alpha }= {\mathrm{span}} (X_{2},X_{3,\alpha},X_{4})$. Then
every shearing subgroup in dimension four is either conjugate to the standard
or the Toeplitz shearing subgroup, or to precisely one of
$\mathfrak{s}_{\alpha}$, $\alpha =-1,0,1$. For $\alpha = 0$, the set of
dilation subgroups compatible with $\mathfrak{s}_{\alpha}$ (or equivalently:
the set of infinitesimal generators of such groups) is a one-dimensional
manifold. For $\alpha = \pm 1$, there exists a two-dimensional manifold
of infinitesimal generators to choose from. Thus we obtain five distinct
families of shearlet dilation groups, each parametrized by a continuum
of possible choices for the associated scaling subgroup, and no two distinct
elements from any of these families are coorbit equivalent.

We finally present an example of two distinct shearlet dilation groups
that are coorbit equivalent. We let $S_{1}$ denote the shearing subgroup
of the Toeplitz shearlet dilation group in dimension $4$, i.e.,
\begin{equation*}
S_{1} = \left \{ \left (
\begin{array}{c@{\quad}c@{\quad}c@{\quad}c}
1 & t_{2} & t_{3} & t_{4}
\\
0 & 1 & t_{2} & t_{3}
\\
0 & 0 & 1 & t_{2}
\\
0 & 0 & 0 & 1
\end{array}
\right ) : t_{2},t_{3},t_{4} \in \mathbb{R} \right \}~.
\end{equation*}
With
\begin{equation*}
C =\left (
\begin{array}{c@{\quad}c@{\quad}c@{\quad}c}
1 & 0 & 0 & 0
\\
0 & 1 & 1 & 0
\\
0 & 0 & 1 & 0
\\
0 & 0 & 0 & 1
\end{array}
\right )~
\end{equation*}
one then computes for $S_{2} = C^{-1} S_{1} C$ that
\begin{equation*}
S_{2} = \left \{ \left (
\begin{array}{c@{\quad}c@{\quad}c@{\quad}c}
1 & t_{2} & t_{3} & t_{4}
\\
0 & 1 & t_{2} & t_{3}-2 t_{2}
\\
0 & 0 & 1 & t_{2}
\\
0 & 0 & 0 & 1
\end{array}
\right ) : t_{2},t_{3},t_{4} \in \mathbb{R} \right \}~,
\end{equation*}
in particular $S_{1} \neq S_{2}$. Then the associated shearlet groups
$H_{1},H_{2}$ obtained by combining $S_{1},S_{2}$ with the isotropic scaling subgroup
$D = \mathbb{R}^{+} \cdot I_{d}$ are distinct too, but they are coorbit
equivalent by \reftext{Theorem~\ref{thm:char_ce_shearlet}}.

\section{Concluding remarks}

The main purpose of this paper was to {rigorously} establish a connection
between coorbit spaces, decomposition spaces and coarse geometry, and to
demonstrate the potential of the method, using the class of shearlet coorbit
spaces. The treatment of the examples has made clear that the coarse reformulation
provides a shorthand for the treatment of potentially fairly technical
questions in the theory of decomposition spaces. To quote from the abstract
of \cite{VoigtlaenderEmbeddingsOfDecompositionSpaces}:
\textit{In a nutshell, although knowledge of Fourier analysis is required
to define and understand decomposition spaces, no such knowledge is required
if one just wants to apply the embedding results presented in this article.
Instead, one only has to study the \textbf{geometric properties of the involved
coverings}, so that one can decide the finiteness of certain sequence space
norms defined in terms of the coverings.} (Boldface added by the authors of the present paper.) The results of our paper indicate
that the pertinent geometric properties mentioned in this quote are \emph{coarse}
properties. While we put our focus on the fundamental question of classification
(deciding, when certain scales of decomposition spaces coincide), we expect
that the coarse viewpoint will also be useful in connection with other
questions in the theory of decomposition spaces.




\section*{Acknowledgement} We thank the anonymous referee for a meticulous, critical reading of our paper, and for many helpful and clarifying remarks. 

\bibliography{References}
\bibliographystyle{plain}

\end{document}